\documentclass[12pt,reqno]{amsart}
\usepackage{amssymb,amsmath,amsthm,amstext,amsfonts,latexsym}
\usepackage[dvips]{graphicx}
\usepackage{color}

\makeatletter \@addtoreset{equation}{section} \makeatother
\theoremstyle{plain}

\theoremstyle{plain}
\newtheorem{maintheorem}{Theorem}

\newtheorem{theorem}{Theorem }[section]
\newtheorem{proposition}[theorem]{Proposition}
\newtheorem{lemma}[theorem]{Lemma}
\newtheorem{corollary}[theorem]{Corollary}
\theoremstyle{definition} \theoremstyle{remark}
\newtheorem{remark}[theorem]{Remark}

\DeclareMathAlphabet{\mathpzc}{OT1}{pzc}{m}{it}

\newcommand{\NN}{{\mathbb N}}

\newcommand{\ZZ}{{\rm\bf Z}}
\newcommand{\RR}{{\mathbb R}}
\newcommand{\EU}{{\rm\bf S}^3}

\newcommand{{\markov}}{T}

\newcommand{\In}{{\text{In}}}
\newcommand{\Out}{{\text{Out}}}
\newcommand{\Fix}{{\text{Fix}}}
\newcommand{\loc}{{\text{loc}}}

\numberwithin{equation}{section}

\begin{document}

\title[Complete set of invariants for a Bykov attractor]{Complete set of invariants for a Bykov attractor}
\author[M. Carvalho]{Maria Carvalho}
\address{Centro de Matem\'atica da Univ. do Porto,
Rua do Campo Alegre, 687,
4169-007 Porto, Portugal}
\email{mpcarval@fc.up.pt}

\author[A. A. P. Rodrigues]{Alexandre A. P. Rodrigues}
\address{Centro de Matem\'atica da Univ. do Porto,
Rua do Campo Alegre, 687,
4169-007 Porto, Portugal}
\email{alexandre.rodrigues@fc.up.pt}

\date{\today}
\thanks{MC and AR were partially supported by CMUP (UID/MAT/00144/2013), which is funded by FCT with national (MEC) and European structural funds through the programs FEDER, under the partnership agreement PT2020. AR also acknowledges financial support from Program INVESTIGADOR FCT (IF/00107/2015). This work has greatly benefited from AR's visit to Nizhny Novgorod University, supported by the grant RNF 14-41-00044. The authors are grateful to the referee for the careful reading of this manuscript and the valuable comments.}

\begin{abstract} In this paper we consider an attracting heteroclinic cycle made by a $1$-dimensional and a $2$-dimensional separatrices between two hyperbolic saddles having complex eigenva\-lues. The basin of the global attractor exhibits historic behaviour and, from the asymptotic properties of these non-converging time averages, we obtain a complete set of invariants under topological conjugacy in a neighborhood of the cycle. These invariants are determined by the quotient of the real parts of the eigenvalues of the equilibria, a linear combination of their imaginary components and also the transition maps between two cross sections on the separatrices.
\end{abstract}

\maketitle

\setcounter{tocdepth}{1}


\section{Introduction}
The classification of vector fields according to their topological properties is a major concern in the study of dynamical systems, and it has been often addressed in recent years in order to distinguish what seems to be similar dynamical systems and to study the stabi\-lity of their properties. Dimension three is the lowest dimension where one finds chaos for flows, but even in this low-dimensional setting a vast catalogue of exotic dynamical phenomena is already known. For instance, the Lorenz attractor \cite{W1979, GW1979,T2002}, or the homoclinic cycle associated to a saddle-focus studied by Shilnikov \cite{Shilnikov65, Shilnikov67}, or else the spiralling attractors generated by heteroclinic networks that Bykov introduced in \cite{Bykov1999, Bykov2000}. The latter has recently attracted more attention due to the need to understand the global dynamical pro\-perties induced by the non-real eigenvalues at the equilibria and the presence of homoclinic or heteroclinic networks; see \cite{HS2010} for background material, related accounts and references therein.

It is known that a vector field $f$ which exhibits a heteroclinic tangency cannot be structurally stable \cite{Rob1973}. This gives rise to interesting invariants under topological conjugacy and implies the existence of an uncountable number of different conjugacy classes in any small neighborhood of $f$. However, it might be possible to describe all possible conjugacy classes with finitely many independent real parameters, which would provide, in particular, a workable description of the systems near $f$. If this is the case, it is said that the \emph{modulus of stability} of $f$ is finite and equal to the minimum number of parameters.\\

\subsection{Modulus for heteroclinic connections}
Regarding invariants for vector fields on $3$-dimensional manifolds, Dufraine discussed in \cite{Duf2001} vector fields with two saddle-focus equilibria $\sigma_1$ and $\sigma_2$, with eigenvalues
\begin{eqnarray}\label{eq:complex-eigenvalues}
-C_{1} \pm i\,\omega_1 & \quad \quad \text{and}  \quad \quad & E_{1} \\
E_{2} \pm i \, \omega_2 &\quad  \quad \text{and}  \quad \quad & -C_{2} \nonumber
\end{eqnarray}
where
\begin{equation}\label{eq:values}
\omega_1,\,\, \omega_2 > 0, \quad C_1 > E_1 > 0, \quad C_2 > E_2 > 0.
\end{equation}
It is also assumed that the equilibria are connected by their $1$-dimensional invariant manifolds (unstable for $\sigma_1$ and stable for $\sigma_2$). Aware of the invariant $\gamma_1=\frac{C_1}{E_2}$ introduced by Palis in \cite{P1978}, which involves the eigenvalues not related with the connection manifold, the author computes another conjugacy invariant associated to the complex parts of the eigenvalues of the vector field at the equilibria, namely $\omega_1 + \gamma_1\,\omega_2$. Dufraine also claims that these two invariants constitute a complete set if the transition map along the $1$-dimensional connection is a homothety-rotation.

Generalizing this approach, but this time aiming at a complete characterization up to topological equivalence, Bonatti and Dufraine considered in \cite{BD2003} the same setting of \cite{Duf2001} while assuming that the transition map along the $1$-dimensional connection, when expressed in suitable coordinates, is linear, with matrix
$\left(\begin{array}{cc}
1 & 0 \\
0 & \lambda
\end{array} \right)$
for some $\lambda \geq 1$, and conformality constant $c=\frac{1}{2}(\lambda + \frac{1}{\lambda})$. Then they showed that a complete set of invariants for orientation preserving topological equivalence essentially depends on Palis' invariant and $c$. More precisely, given two such vector fields whose transition maps have conformality constants $c$ and $\widetilde{c}$ sufficiently small (a size measured by an adequate continuous function $\psi \geq 1$ which depends on $\omega_1$, $\omega_2$, $C_1$ and $E_2$), then the vector fields are topologically equivalent in a neighborhood of the corresponding connections, and they are positively topologically equivalent if and only if one of the following conditions is satisfied: either
\begin{equation}\label{eq:4th invariant}
c=\widetilde{c}=1 \quad \text{ and } \quad \omega_1 - \gamma_1\,\omega_2=0=\widetilde{\omega_1} - \widetilde{\gamma_1}\,\widetilde{\omega_2}
\end{equation}
or else
$$\left(\omega_1 - \gamma_1\,\omega_2\right)\left(\widetilde{\omega_1} - \widetilde{\gamma_1}\,\widetilde{\omega_2}\right)>0.$$
Although not relevant to our setting, we notice that the authors also proved that, when $c$ and $\widetilde{c}$ are both too big, then a complete set of invariants is made by the quotient $\frac{\omega_1}{\gamma_1\,\omega_2}$ and the value at $(\frac{\omega_1}{C_1}, \, \frac{\omega_2}{E_2},\, c)$ of a function related to $\psi$.

Finally, in \cite{SSimo1989}, Sim\'o and Sus\'in analyzed vector fields with a heteroclinic $2$-dimensional connection between $\sigma_1$ and $\sigma_2$, and proved that the classes of conjugacy of these systems can be characterized by one parameter only, depending on the eigenvalues of the equilibria which are not related to the two-dimensional manifold, namely
$$\gamma_2 = \frac{C_2}{E_1}.$$

\subsection{Modulus for heteroclinic cycles}
Concerning heteroclinic cycles, for planar vector fields and based on Bowen's example, Takens described in \cite{Takens1994} a complete set of topological invariants under topological conjugacies for attracting heteroclinic cycles with two $1$-dimensional connections between two hyperbolic saddles with real eigenvalues $-C_{1} < 0$, $E_{1}>0$ and $-C_2 <0$, $E_{2} > 0$, respectively. Takens assumes that the transition maps on $1$-dimensional cross sections are linear, namely the identity map and $x \mapsto a\,x$ for some $0<a<1$. The set of the three invariants characterized by Takens includes, as expected, the ones previously reported by Palis in \cite{P1978} associated to each $1$-dimensional connection of the cycle, that is,
$$\gamma_1=\frac{C_1}{E_2} \quad \text{ and } \quad \gamma_2=\frac{C_2}{E_1}.$$
In addition to these, Takens found another invariant which is primarily determined by the transition maps and is given by
$$\frac{1}{E_1}(1+\gamma_1)\,\log a.$$
Takens' construction of the conjugacy uses asymptotic properties of non-converging Birkhoff time averages, the so called historic behavior \cite{Ruelle2001}, which we will recall in Subsection~\ref{ssec:hist.behavior}.

\subsection{Modulus of stability for a Bykov attractor}
In this work, we will consider vector fields acting on the $3$-dimensional sphere with two hyperbolic equilibria admitting complex eigenvalues satisfying \eqref{eq:complex-eigenvalues} and \eqref{eq:values}, and exhibiting both a one-dimensional separatrix as in \cite{Duf2001} and a two-dimensional connection like the one in \cite{SSimo1989}. This network is called a Bykov attractor.
 The behavior of such a vector field $f$ in the vicinity of the connections is essentially given, up to conjugacy, by the linear part of $f$ in linearizing neighborhoods of the equilibria and by the transition maps between two discs transversal to $f$ and contained in those neighborhoods. Therefore, we expect to revisit the two invariants described in \cite{Duf2001}, the one introduced in \cite{SSimo1989} and the invariant Takens found in \cite{Takens1994}.

Following the strategy of Takens \cite{Takens1994}, we will select suitable $2$-dimensional cross sections at the connections and assume, as done in \cite{Takens1994}, \cite{BD2003} and \cite{Bykov1999}, that in appropriate coordinates 
the transition maps are linear, whose $2\times2$ matrices are diagonal and have non-zero determinants, namely $\left(\begin{array}{cc}
\frac{1}{a} & 0 \\
0 & a
\end{array} \right)$ and
$\left(\begin{array}{cc}
1 & 0 \\
0 & \lambda
\end{array} \right)$, for some $0 < a < 1$ and some $\lambda \geq 1$. We will assume, as done in \cite{Takens1994} for a heteroclinic connection between two saddles with real eigenvalues, that the transition along the one-dimensional connection is the identity map (that is, $\lambda = 1$) for all the vector fields under consideration. This simplifies the computations without demanding more invariants. Indeed, according to Bonatti and Dufraine's result \eqref{eq:4th invariant}, only if we are looking for an orientation preserving topological equivalence is it required that either $\frac{\omega_1}{\gamma_1 \omega_2}$ is equal to $1$ for each conjugated pair of vector fields or the sign of $\omega_1-\gamma_1\,\omega_2$ is identical for both vector fields. We observe that, since $\gamma_1$ and $\omega_1 + \gamma_1 \omega_2$ are invariants by conjugacy (cf. \cite{Duf2001}), using the equality
$$\omega_1 + \gamma_1 \omega_2 = \omega_1 \left(1 + \frac{\gamma_1 \omega_2}{\omega_1}\right)$$
we conclude that the value of $\frac{\omega_1}{\gamma_1 \omega_2}$ is equal for two conjugate Bykov attractors if and only if the values of $\omega_1$ and $\omega_2$ are the same for the two vector fields. This outcome resembles the topological invariants found by Dufraine in \cite{Duf2001} for homoclinic orbits of saddle-focus type, the so called Shilnikov cycles.

We will also suppose, as in \cite{Takens1994}, that the transitions along the one-dimensional and two-dimensional connections happen instantaneously. This is a reasonable assumption due to the fact that, as the Bykov attractor is asymptotically stable (cf. \cite{KM1995}), if $P$ belongs to its basin then the period of time spent by the orbit $\left(\varphi(t, P)\right)_{t \in \mathbb{R}^+}$ inside small neighborhoods of $\sigma_1$ and $\sigma_2$ tends to infinity as $t \to + \infty$, whereas the time used to travel between these two neighborhoods remains uniformly bounded in length. Afterwards, we will analyze the sequence of hitting times of each orbit at chosen cross sections and show that a complete set of invariants under topological conjugacy for these Bykov vector fields is given by
$$\gamma_1=\frac{C_1}{E_2}, \quad \gamma_2=\frac{C_2}{E_1}, \quad \omega_1 + \gamma_1\,\omega_2, \quad \frac{1}{E_1}(1+\gamma_1)\,\log a.$$

\section{Description of the vector fields}\label{se:setting}

Let $f: \EU \to \mathbb{R}^4$ be a $C^r$, $r \geq 3$, vector field on the $C^\infty$ Riemannian $3$-dimensional differential manifold
$$\EU=\{(x_1, x_2, x_3, x_4) \in \RR^4 \colon x_1^2 + x_2^2 + x_3^2 + x_4^2 = 1\}$$
whose corresponding flow is given by the unique solutions $t \in \mathbb{R} \mapsto x(t)=\varphi(t,x_{0})\in \EU$ of the initial-value problem
$$\left\{\begin{array}{l}
\dot{x}(t) = f(x(t)) \\
x(0) = x_{0}.
\end{array}
\right.
$$
We will request that the organizing center of $f$ satisfies the following conditions: \\

\begin{enumerate}
\item[P1.] The vector field $f$ is equivariant under the action of $\ZZ_2 \oplus \ZZ_2$ on $\EU$ induced by the two linear maps on $\RR^4$
\begin{eqnarray*}
\Gamma_1(x_1,x_2,x_3,x_4) &=& (-x_1,-x_2,x_3,x_4) \\
\Gamma_2(x_1,x_2,x_3,x_4) &=& (x_1,x_2,-x_3,x_4).
\end{eqnarray*}
That is, $f \circ \Gamma_1 = \Gamma_1 \circ f$ and $f \circ \Gamma_2 = \Gamma_2 \circ f$.
\vspace{0.3cm}
\item[P2.] The set $\Fix\,(\ZZ_2 \oplus \ZZ_2)=\{Q \in \EU:\Gamma_1(Q)=\Gamma_2(Q)= Q\}$ reduces to two equilibria, namely $\sigma_1=(0,0,0,1)$ and $\sigma_2=(0,0,0,-1)$, which are hyperbolic saddle-foci whose eigenvalues are, respectively,
\begin{eqnarray*}
-C_{1} \pm \omega_1 \,i &\quad \quad \text{and}  \quad \quad & E_{1} \\
E_{2} \pm \omega_2 \,i &\quad  \quad \text{and}  \quad \quad& -C_{2}
\end{eqnarray*}
where
\begin{equation}\label{eq:eigenvalues}
\omega_1 > 0, \quad \omega_2 > 0, \quad C_1 > E_1 > 0, \quad  C_2 > E_2 > 0.
\end{equation}
\vspace{0.2cm}
\item[P3.] The flow-invariant circle $\Fix\,(\ZZ_2(\Gamma_{1}))=\{Q \in \EU:\Gamma_1(Q) = Q\}$ consists of the two equilibria, $\sigma_1$ and $\sigma_2$, and two heteroclinic trajectories from $\sigma_1$ to $\sigma_2$ we denote by $[\sigma_1 \rightarrow \sigma_2]^{ext}$ and $[\sigma_1 \rightarrow \sigma_2]^{int}$, whose union will be simply called $[\sigma_1 \rightarrow \sigma_2]$.
\vspace{0.5cm}
\item[P4.] The flow-invariant sphere $\Fix\,(\ZZ_2(\Gamma_{2}))=\{Q \in \EU:\Gamma_2(Q) = Q \}$ is made of the two equilibria $\sigma_1$ and $\sigma_2$ and a two-dimensional heteroclinic connection from $\sigma_2$ to $\sigma_1$.
\vspace{0.5cm}
\item[P5.] The saddle-foci $\sigma_1$ and $\sigma_2$ have the same chirality (which means that near $\sigma_1$ and $\sigma_2$ all trajectories turn in the same direction around the one-dimensional connections $[\sigma_1 \rightarrow \sigma_2]^{ext}$ and $[\sigma_1 \rightarrow \sigma_2]^{int}$; see \cite{LR2015} for more information).
\vspace{0.5cm}
\end{enumerate}
We denote by $\mathfrak{X}^r_{\text{Byk}}(\EU)$ the set of $C^r$, $r\geq 3$, smooth $\ZZ_2\oplus \ZZ_2$--equivariant vector fields on $\EU$ that satisfy the assumptions (P1)--(P5), endowed with the $C^r$--Whitney topology.  Figure~\ref{cycle_scheme} illustrates the previous information concerning $\Fix\,(\Gamma_1)$ and $\Fix\,(\Gamma_2)$.

\begin{figure}[h]
\begin{center}
\includegraphics[height=6cm]{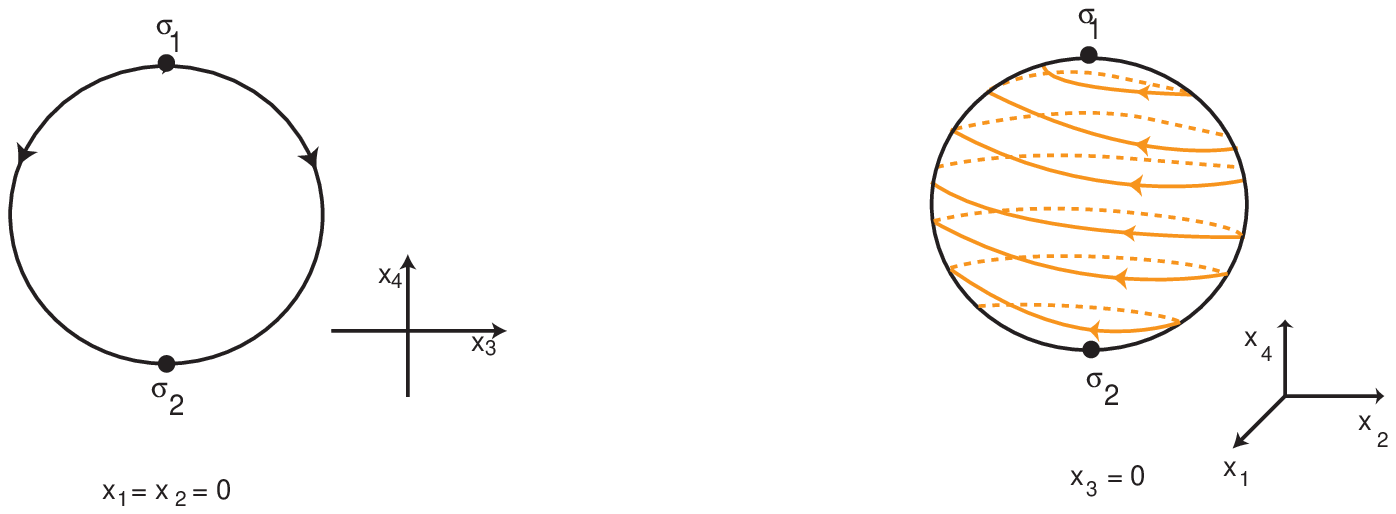}
\end{center}
\caption{\small Heteroclinic connections of the organizing center: (a) The invariant circle $\Fix\,(\Gamma_1)$, parameterized in $\EU$ by $x_3^2 + x_4^2 = 1$, consists of $\sigma_1$ and $\sigma_2$ and two trajectories connecting them. (b) The invariant sphere $\Fix\,(\Gamma_2)$, parameterized in $\EU$ by $x_1^2 + x_2^2 +x_4^2 = 1$, is a two-dimensional connection between $\sigma_2$ and $\sigma_1$.}
\label{cycle_scheme}
\end{figure}

The two equilibria $\sigma_1$ and $\sigma_2$, the two trajectories listed in (P3) and the two-dimensional heteroclinic connection from $\sigma_2$ to $\sigma_1$ refered to in (P4) build a heteroclinic network we will denote hereafter by $\mathcal{A}_f$. This set is the global Bykov attractor of the dynamical system $f$ in an open set $V^0 \subset \EU$. More precisely, the heteroclinic network $\mathcal{A}_f$ is asymptotically stable, that is, there exists an open neighborhood $V^0$ of $\mathcal{A}_f$ in $\EU$  such that every solution starting in $V^0$ remains inside $V^0$ for all positive times and  is forward asymptotic to the network $\mathcal{A}_f$.

\begin{remark}
The assumptions (P1)--(P5) define a degenerate set of vector fields exhibiting heteroclinic connections, a dynamical phenomenon which is natural within systems with symmetry. If we slightly perturb a vector field in $\mathfrak{X}^r_{\text{Byk}}(\EU)$ in order to break the connection $\Fix(\ZZ_2(\Gamma_2))$ while preserving the $\Gamma_1$--equivariance, generically we obtain what the authors of \cite{LR2015} call a Bykov cycle, with saturated horseshoes accumulating on it. Finding a complete set of invariants for these Bykov cycles is still an open problem.
\end{remark}

\section{Main definitions}\label{se:definitions}

\subsection{Invariants under conjugacy}\label{ssec:Modulus}
Given two systems $\dot{x} = f_1(x)$ and $\dot{x} = f_2(x)$, defined in domains $D_1\subset \EU$ and $D_2\subset \EU$, respectively, let $\varphi_i(t, x_0)$ be the unique solution of $\dot{x}=f_i(x)$ with initial condition $x(0)=x_0$, for $i \in \{1,2\}$. We say that the corresponding flows are topologically equivalent in subregions $U_1 \subset D_1$ and $U_2 \subset D_2$ if there exists a homeomorphism $h: U_1 \rightarrow U_2$ which maps solutions of the first system onto solutions of the second preserving the time orientation. If $h$ is also time preserving, that is, if for every $x \in \EU$ and every $t \in \RR$, we have $\varphi_1(t, h(x))=h(\varphi_2(t,x))$, the flows are said to be \emph{topologically conjugate} and $h$ is called a \emph{topological conjugacy}.

A functional $\mathcal{I}$ defined on a set $\mathcal{V}$ of vector fields is an \emph{invariant under topological conjugacy} if, whenever two vector fields $f$ and $\tilde{f}$ in $\mathcal{V}$ are conjugate, then $\mathcal{I}(f)=\mathcal{I}(\tilde{f})$. A set of invariants under topological conjugacy is said to be \emph{complete} if, given two systems with equal invariants, there exists a topological conjugacy between the corresponding flows. For the sake of simplicity, in what follows we will talk about a topological conjugacy between two vector fields while meaning a conjugacy defined on some neighbourhood of the global attractors of the associated flows.

\subsection{Historic behavior}\label{ssec:hist.behavior}
We say that the orbit of a point $P$ by a flow $\varphi: \mathbb{R} \times X \to X$ has historic behavior if, for some continuous function $G: X \to \mathbb{R}$, the Birkhoff averages of $G$ along the orbit of $P$
$$\left(\frac{1}{t}\int_0^t \,G (\varphi(s,P))\,ds\right)_{t \, \in \, \mathbb{R}^+}$$
does not converge.

\section{Constants}\label{se:constants}

For future use, we settle the following notation:
\begin{equation*}
\gamma_1 = \frac{C_1}{E_2} \qquad \qquad \gamma_2 = \frac{C_2}{E_1} \qquad \qquad \delta_1 = \frac{C_1}{E_1} \qquad \qquad \delta_2 = \frac{C_2}{E_2}
\end{equation*}
\begin{equation*}
K = \frac{1}{E_1}\,\left(\omega_1 + \gamma_1\,\omega_2\right) \qquad \qquad \tau = \frac{1}{E_1}\,(1+\gamma_1) \qquad \qquad \delta = \gamma_1\gamma_2 = \delta_1\delta_2 = \frac{C_1C_2}{E_1 E_2}.
\end{equation*}
\bigbreak
According to the assumptions stated in \eqref{eq:eigenvalues}, we have $K > 0$, $\tau > 0$, $\delta_1 > 1$ and $\delta_2 > 1$.

\section{Main results}\label{se:results}
There are well-known examples of dynamical systems with interesting sets of orbits exhibiting historic behavior, as the logistic family \cite{HK1990}, the example of Bowen \cite{Takens1994}, the full shifts on finite symbols \cite{HK1990, Takens2008, BS2000, CV2001, KS2017}, Gaunersdorfer's systems on a simplex \cite{Gau1992} and the Lorenz attractor \cite{KLS2016}. Following the arguments of Takens in \cite{Takens1994} and Gaunersdorfer in \cite{Gau1992}, we will add to this list another example in Proposition~\ref{thm:A}. Indeed, although the points in the attractor $\mathcal{A}_f$ do not have historic behavior since the $\omega$-limit and $\alpha$-limit of their orbits are either $\sigma_1$ or $\sigma_2$, all elements in its proper basin of attraction
$$\mathcal{B}(\mathcal{A}_f)=\{P \in \EU \colon \,\, \text{the accumulation points of} \,\, \left(\varphi(t,P)\right)_{t \,\in\, \mathbb{R}^+} \text{ belong to } \mathcal{A}_f\} \setminus \mathcal{A}_f$$
display this kind of behavior.

\begin{proposition}\label{thm:A} Let $f$ be a vector field in $\mathfrak{X}^r_{\text{Byk}}(\EU)$. Given a continuous map $G\colon \EU \to \mathbb{R}$ and $P \in \mathcal{B}(\mathcal{A}_f)$, there exist a sequence $(t_k)_{k\,\in\, \NN}$, which depend on $P$ and $G$, such that
$$\lim_{i\,\to \, +\infty} \,\frac{1}{t_{2i}} \,\int_0^{t_{2i}} G(\varphi(t,\,P))\,dt = \frac{1}{1+ \gamma_1} \,G(\sigma_1) + \frac{\gamma_1}{1+ \gamma_1} \,G(\sigma_2)$$
and
$$\lim_{i\,\to\, +\infty}\, \frac{1}{t_{2i+1}}\, \int_0^{t_{2i+1}} G(\varphi(t,\,P))\,dt = \frac{\gamma_2}{1 + \gamma_2}\, G(\sigma_1) + \frac{1}{1 + \gamma_2} \,G(\sigma_2).$$
\medskip

\noindent Consequently, every point in $\mathcal{B}(\mathcal{A}_f)$ has historic behavior.
\end{proposition}

As happens in the context of Bowen's planar vector fields \cite{Takens1994}, the previous result is the source of four invariants under topological conjugacy, generated by the sequences of hitting times at two cross sections appropriately chosen in a neighborhood of the attractor $\mathcal{A}_f$. We will also show that they build a complete set of invariants.

\begin{maintheorem}\label{thm:B} Let $f$ be a vector field in $\mathfrak{X}^r_{\text{Byk}}(\EU)$. Then $\gamma_1$, $\gamma_2$, $\omega_1+\gamma_1\omega_2$ and $\tau \log a$ form a complete set of invariants for $f$ under topological conjugacy.
\end{maintheorem}

\section{Dynamics in $\mathcal{B}(\mathcal{A}_f)$}\label{se:localdyn}

We will analyze the dynamics near the network $\mathcal{A}_f$ of a vector field $f \in \mathfrak{X}^r_{\text{Byk}}(\EU)$, $r \geq 3$, using local maps that approximate the dynamics near and between the two equilibria in the network.

\subsection{Local coordinates}\label{sse:local-coordinates}

In order to describe the dynamics in a neighborhood of the Bykov cycle $\mathcal{A}_f$ of $f$, we need a workable expression of the Poincar\'e map of $f$ at a suitable cross section inside this neighborhood. For that, up to a conjugacy, we will select the linearizing coordinates near the equilibria $\sigma_1$ and $\sigma_2$ introduced in \cite{Deng1989}. In these coordinates, and after a linear rescaling of the variables, if needed, we consider two cylindrical neighborhoods $V_{\sigma_1}$ and $V_{\sigma_2}$  in ${\EU}$ of $\sigma_1 $ and $\sigma_2 $, respectively, with base-radius $1$ and height $2$ (see Figure \ref{crosssections1}). Moreover, the boundaries of $V_{\sigma_1}$ and $V_{\sigma_2}$ have three components:
\begin{enumerate}
\item The cylinder wall, parameterized in cylindrical coordinates $(\rho ,\theta ,z)$ by
$$\theta \,\in \,[0,2\pi], \,\,|z|\leq 1 \quad \mapsto \quad (1, \theta, z).$$
\item Two discs, the top and bottom of each cylinder, parameterized by
$$\rho \, \in \, [0,1], \,\, \theta \,\in \,[0, 2\pi] \quad \mapsto \quad (\rho,\theta,\pm 1).$$
\end{enumerate}

Observe that the local stable manifold $W^s_\loc(\sigma_1)$ of the equilibrium $\sigma_1$ is the disk in $V_{\sigma_1}$ given by $\{(\rho,\theta,z)\colon 0\leq \rho \leq 1, \, 0\leq \theta \leq 2\pi, \, z=0\}$. The local stable manifold $W^s_\loc(\sigma_2)$ of $\sigma_2$ is the $z$-axis in $V_{\sigma_2}$; the local unstable manifold $W^u_\loc(\sigma_2)$ of $\sigma_2$ is parameterized by $z=0$ in $V_{\sigma_2}$.

\begin{figure}[ht]
\begin{center}
\includegraphics[height=6cm]{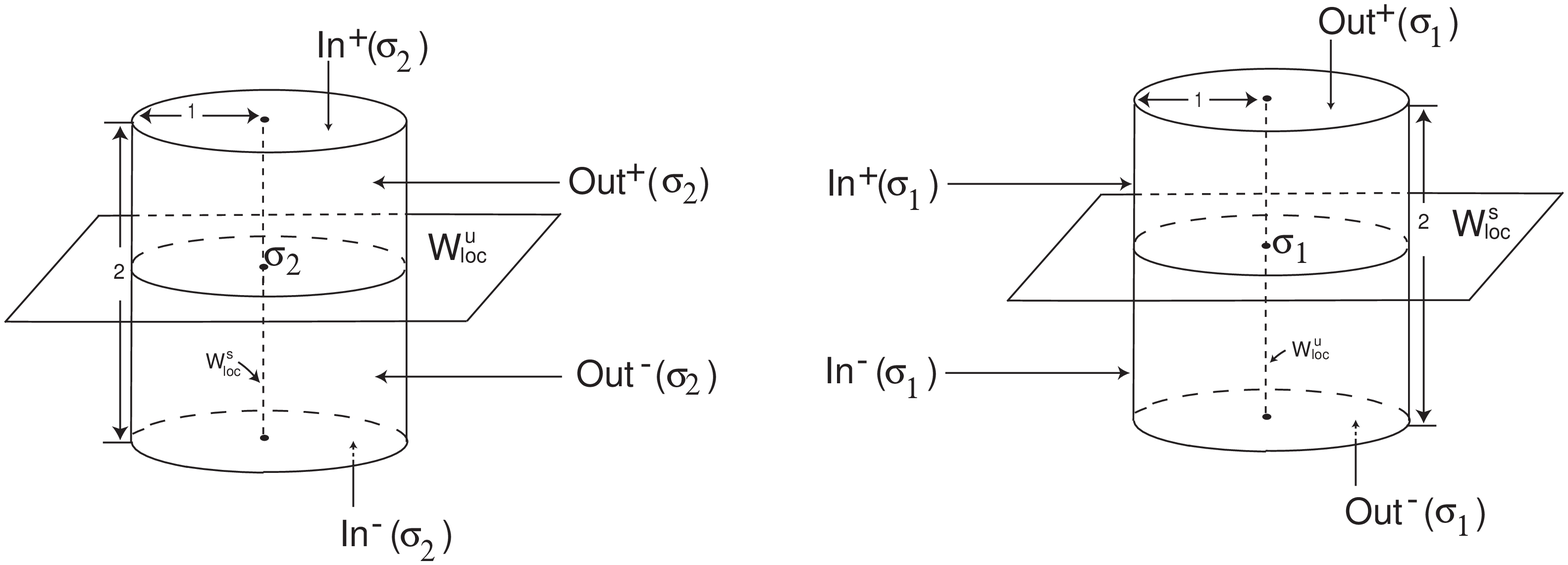}
\end{center}
\caption{\small Local cylindrical coordinates in $V_{\sigma_2}$ and $V_{\sigma_1}$, near $\sigma_2$ and $\sigma_1$. Left: The flow enters the cylinder $V_{\sigma_2}$ crossing transversely the top/bottom $\In\,(\sigma_2)$ and leaves it transversely through the wall $\Out\,(\sigma_2)$. Right: The flow enters the cylinder $V_{\sigma_1}$ transverse to the wall $\In\,(\sigma_1)$ and leaves it transversely across the top/bottom $\Out\,(\sigma_1)$.}
\label{crosssections1}
\end{figure}

In $V_{\sigma_1}$, we will use the following notation:
\begin{eqnarray*}
\Out^+(\sigma_1) &:=& \text{top of $V_{\sigma_1}$, that is, $\left\{(\rho,\, \theta,\, 1) \colon 0 \leq \rho < 1, \,\, \theta \,\in \,[0, 2\pi]\,\right\}$.}\\
\Out^-(\sigma_1) &:=& \text{bottom of $V_{\sigma_1}$, that is, $\left\{(\rho,\, \theta,\, -1) \colon 0 \leq \rho < 1, \,\, \theta \,\in \,[0, 2\pi]\,\right\}$.}\\
\In^+(\sigma_1) &:=& \text{upper part of the cylinder wall of $V_{\sigma_1}$, that is, the set} \\
&& \left\{(1,\, \theta, \,z) \colon \theta \,\in \,[0, 2\pi], \,\, 0 < z < 1\,\right\}. \\
\In^-(\sigma_1) &:=& \text{lower part of the cylinder wall of $V_{\sigma_1}$, that is, the set} \\
&& \left\{(1, \,\theta, \,z) \colon \theta \,\in \,[0, 2\pi], \,\, -1 < z < 0\,\right\} .\\
\In\,(\sigma_1) &:=&  \In^+(\sigma_1) \cup \In^-(\sigma_1); \text{ its elements enter $V_{\sigma_1}$ in positive time.}\\
\Out\,(\sigma_1) &:=& \Out^+(\sigma_1) \cup \Out^-(\sigma_1); \text{ its elements leave $V_{\sigma_1}$ in positive time.}\\
\end{eqnarray*}
By construction, the flow is transverse to these sections. Similarly, we define the cross sections for the linearization around $\sigma_2$. We will refer to $\In\,(\sigma_2)$, the top and the bottom of $V_{\sigma_2}$, consisting of points that enter $V_{\sigma_2}$ in positive time; $\Out\,(\sigma_2)$, the cylinder wall of $V_{\sigma_2}$, made of points that go inside $V_{\sigma_2}$ in negative time, with $\Out^+(\sigma_2)$ denoting its upper part and $\Out^-(\sigma_2)$ its lower part. 
Notice that $[\sigma_1\rightarrow\sigma_2]$ connects points with $z>0$ in $V_{\sigma_1}$ (respectively $z<0$) to points with $z>0$ (respectively $z<0$) in $V_{\sigma_2}$.

Let
\begin{equation}\label{eq:cross-sections}
\Sigma_1 = \Out\,(\sigma_1) \qquad \text{ and } \qquad \Sigma_2 = \Out\,(\sigma_2)
\end{equation}
be two relative-open cross sections transverse to the connections $[\sigma_1 \rightarrow \sigma_2]^{ext}$ and the invariant sphere $\Fix\,(\Gamma_2)$, respectively. Geometrically each connected component of $\Sigma_1 \setminus [\sigma_1 \rightarrow \sigma_2]$ is a punctured disk; on the other hand, $\Sigma_2 \setminus \Fix\,(\Gamma_2)$ is an annulus, as illustrated in Figure~\ref{cross-sections1}.

\begin{figure}[ht]
\begin{center}
\includegraphics[height=7cm]{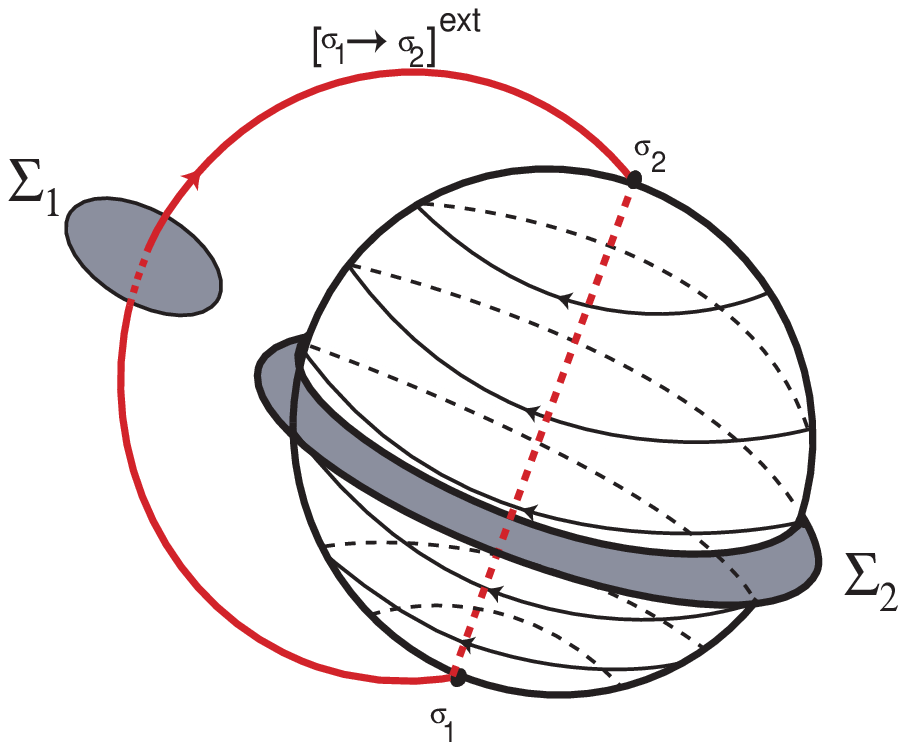}
\end{center}
\caption{$\Sigma_1$ and $\Sigma_2$.} 
\label{cross-sections1}
\end{figure}

\subsection{Local maps near the saddle-foci}\label{sse:local saddle-foci}

As the dynamics sends points with $z>0$ in $V_{\sigma_1}$ (respectively $z<0$) to points with $z>0$ (respectively $z<0$) in $V_{\sigma_2}$, and is symmetric with respect to the two-dimensional sphere $\Fix\,(\Gamma_2)$, it is enough to analyze the orbits of initial conditions $(\rho, \theta, z)$ in the invariant upper part $z > 0$.

In cylindrical coordinates $(\rho ,\theta ,z)$ the linearization of the dynamics at $\sigma_1 $ and $\sigma_2 $ is specified by the following equations:
\begin{equation}\label{local map v}
\left\{\begin{array}{l}
\dot{\rho}=-C_{1 }\rho \\
\dot{\theta}=\omega_1 \\
\dot{z}=E_1 z
\end{array}
\right.
\qquad\text{and}\qquad
\left\{\begin{array}{l}
\dot{\rho}=E_{2 }\rho \\
\dot{\theta}= \omega_2 \\
\dot{z}=-C_{2 }z .
\end{array}
\right.
\end{equation}
Therefore, a trajectory whose initial condition is $(1,\theta,z) \in \In^+(\sigma_1)$ arrives at $\Out^+(\sigma_1)$ after a period of time equal to
\begin{equation}\label{time1}
-\,\frac{\log z}{E_1}.
\end{equation}
Moreover, the trajectory of such a point $(1,\theta,z)$ leaves $V_{\sigma_1}$ through $\Out^+(\sigma_1)$ at the point of $\Out^+(\sigma_1)$ given by
\begin{equation}\label{local_v}
\Phi^+_{1}\left(1,\theta,z\right)=\left(z^{\delta_1} + \mathcal{S}_1(1,\,\theta, \,z),\,\,\left[-\,\frac{\omega_1}{E_1}\log z + \theta + \mathcal{S}_2(1,\,\theta, \,z)\right] \,\,mod\,2\pi, \,\,1\right)
\end{equation}
where $\mathcal{S}_1$, $\mathcal{S}_2$ are smooth functions such that, for $i = 1,2$ and every $k, \ell \in \NN_0$,
\begin{equation}\label{diff_res}
\left| \frac{\partial^{k+\ell}\,\mathcal{S}_i(1, \,\theta, \,z)}{\partial \theta^k \,\partial z^\ell}
\right| 
= \mathcal{O}(z^{\delta_1 \, (1 + \varepsilon)})
\end{equation}
for some positive constant $\varepsilon < 1$ (cf. \cite{Deng1989}). Dually, a point $(\rho, \theta, 1) \in \In^+(\sigma_2)$ leaves $V_{\sigma_2}$ through $\Out^+(\sigma_2)$ after a period of time equal to
\begin{equation}
\label{time2}
-\,\frac{\log \rho}{E_2}
\end{equation}
at the point of $\Out^+(\sigma_2)$
\begin{equation}\label{local_w}
\Phi^+_{2}\left(\rho, \theta, 1\right)=\left(1, \,\,\left[-\,\frac{\omega_2}{E_2}\log \rho + \theta + \mathcal{T}_1(\rho, \,\theta, \,1)\right]\,\,mod\,2\pi,\,\,\rho^{\delta_2} + \mathcal{T}_2(\rho, \,\theta, \,1)\right)
\end{equation}
where $\mathcal{T}_1$, $\mathcal{T}_2$ satisfy a condition similar to (\ref{diff_res}). The maps $\mathcal{S}_1$, $\mathcal{S}_2$, $\mathcal{T}_1$ and $\mathcal{T}_2$ represent asymptotically small terms that vanish when either $\rho$ or $z$ goes to zero.

\subsection{The first transition}\label{sse:1st transition}

Points in $\Out^+(\sigma_1)$ near $W^u(\sigma_1)$ are mapped into $\In^+(\sigma_2)$ along a flow-box around the connection $[\sigma_1\rightarrow\sigma_2]^{ext}$. Up to a change of coordinates (see \cite[\S 3.1]{Takens1994}, \cite[\S 1.1]{BD2003}), corresponding to homotheties and rotations which leave invariant the local expressions of the flows in neighborhoods of the equilibria, we may assume that the transition map
$$\Psi^+_{1 \rightarrow 2}: \Out^+(\sigma_1) \rightarrow \In^+(\sigma_2)$$
is the identity map. This assumption is compatible with the hypothesis (P3), and admissible due to the fact that the equilibria $\sigma_1$ and $\sigma_2$ have the same chirality (cf. (P5)).

Denote by $\eta^+$ the map
$$\eta^+=\Phi^+_{2} \circ \Psi^+_{1 \rightarrow 2} \circ \Phi^+_{1 }: \In^+(\sigma_1) \rightarrow \Out^+(\sigma_2).$$
Up to higher order terms (we will replace by dots), from \eqref{local_v} and \eqref{local_w} we infer that the expression of $\eta^+$ in local coordinates is given by
\begin{equation}\label{eq:eta}
\eta^+(1,\theta, z)=\left(1,\,\, \left[-\,\frac{\omega_1E_2+\delta_1\omega_2E_1}{E_1E_2} \log z + \theta\right] \,\,mod\,2\pi, \,\,z^{\delta_1\delta_2} \right) + (\ldots)
\end{equation}
or, in a simpler notation (see Section~\ref{se:constants}),
\begin{equation}
\eta^+(1,\theta, z)=\left(1,\,\, (-K \log z + \theta) \,\,mod\,2\pi, \,\,z^{\delta} \right) + (\ldots).
\end{equation}

\subsection{The second transition}\label{sse:2nd transition}

Up to the previously mentioned change of coordinates, the action of the linear part of the transition map
$$\Psi^+_{2 \rightarrow 1}: \Out^+(\sigma_2) \rightarrow \In^+(\sigma_1)$$
may be seen (cf. \cite{Bykov2000}) as a composition of transpositions (either rotations of the coordinate axes or homothetic changes of scales). In what follows, we will assume that
\begin{equation}\label{transition2}
\Psi^+_{2 \rightarrow 1} (1,\theta, z) = \left(1, \,\ \frac{1}{a}\,\theta , \,\, a\,z\right)
\end{equation}
for some $0 < a < 1$.

\subsection{The Poincar\'e map}\label{sse:1st return map}
The $f-$solution of every initial condition $(1, \theta, z) \in \mathcal{B}(\mathcal{A}_f) \cap \In^+(\sigma_1)$ returns to $\In^+(\sigma_1)$, thus defining a first return map
\begin{equation}\label{first return}
\mathcal{P}^+ = \Psi^+_{2 \rightarrow 1} \circ \eta^+: \mathcal{B}(\mathcal{A}_f) \cap \In^+(\sigma_1) \rightarrow \In^+(\sigma_1)
\end{equation}
which is as smooth as the vector field $f$ and acts as follows:
$$\mathcal{P}^+(1, \theta, z) = \left(1,\,\,\frac{1}{a}\,\left[(-K \log z + \theta) \,\,mod\,2\pi\right],\,\, a\,z^\delta \right)  + (\ldots).$$
In an analogous way, we define the return map $\mathcal{P}^-$ from $ \In^-(\sigma_1)$ to itself.

\section{Hitting times}\label{se:hitting-times}

In what follows, we will obtain estimates of the amount of time a trajectory spends between consecutive isolating blocks. We are assuming, as done by Gaunersdorfer \cite{Gau1992} and Takens in \cite{Takens1994}, that the transitions from $\Out^+(\sigma_2)$ to $\In^+(\sigma_1)$ and from $\Out^+(\sigma_1)$ to $\In^+(\sigma_2)$ are instantaneous.

Starting with the initial condition $(1, \theta_0, z_0) \in \Out^+(\sigma_2)$ at the time $t_0$, its orbit hits $\Out^+(\sigma_1)$ after a time interval equal to
\begin{equation}\label{eq:t1}
t_1 = -\,\frac{1}{E_1} \log \,(a\,z_0)
\end{equation}
at the point
$$(\rho_1, \, \theta_1, \,1) = \Phi^+_{1} \circ \Psi^+_{2 \rightarrow 1} (1,\,\theta_0, \,z_0) = \Phi^+_{1} \left(1, \,\, \frac{1}{a}\,\theta_0, \,\, a\,z_0\right)$$
that is,
\begin{eqnarray*}
&&(\rho_1, \, \theta_1, \,1) =  \\
&=& \left(\left(a\,z_0\right)^{\delta_1} + \mathcal{S}_1(1, \,\frac{1}{a}\,\theta_0, \,a\,z_0), \,\,\left[-\, \frac{\omega_1}{E_1} \log \left(a\,z_0\right)+\frac{1}{a}\,\theta_0 + \mathcal{S}_2(1, \,\frac{1}{a}\,\theta_0,\, \,\,a\,z_0)\right]\,\,mod\,2\pi, \,\,1\right).
\end{eqnarray*}
Then, the orbit goes instantaneously to $\In^+(\sigma_2)$ and proceeds to $\Out^+(\sigma_2)$, hitting the point
$$(1, \,\theta_2, \, z_2)=\Phi^+_{2}(\rho_1, \, \theta_1, \,1) \, \in \, \Out^+(\sigma_2)$$
whose second and third coordinates are (see \eqref{local_w})
\begin{eqnarray*}
\theta_2 &=& \left(-\,\frac{\omega_2}{E_2}\log \rho_1 + \theta_1 + \mathcal{T}_1(\rho_1, \,\theta_1, \,1)\right)\,\,mod\,2\pi \\
&=& \left[-\,\frac{\omega_2}{E_2}\log \left(\left(a\,z_0\right)^{\delta_1} + \mathcal{S}_1(1, \,\frac{1}{a}\,\theta_0, \,a\,z_0)\right) + \theta_1 + \mathcal{T}_1(\rho_1, \,\theta_1, \,1)\right]\,\,mod\,2\pi \\ \\
z_2 &=& \rho_1^{\delta_2} + \mathcal{T}_2(\rho_1, \,\theta_1, \,1) \\
&=& \left(\left(a\,z_0\right)^{\delta_1} + \mathcal{S}_1(1, \,\frac{1}{a}\,\theta_0, \,a\,z_0)\right)^{\delta_2} + \mathcal{T}_2(\rho_1, \,\theta_1, \,1)
\end{eqnarray*}
and spending in the whole path a time equal to
\begin{equation*}
t_2 = t_1 + \left(-\,\frac{1}{E_2} \log \rho_1\right)
\end{equation*}
that is,
\begin{equation}\label{eq:t2}
t_2 = t_1 -\,\frac{1}{E_2} \log \left(\left(a\,z_0\right)^{\delta_1} + \mathcal{S}_1(1, \,\frac{1}{a}\,\theta_0, \,a\,z_0)\right) = t_1 - \frac{\delta_1}{E_2} \log \,(a\,z_0) + \mathcal{O}((az_0)^{\delta_1 \,\varepsilon}).
\end{equation}
And so on for the other time values.\\

We may assume, without loss of generality, that $t_0=0$: this amounts to consider the solutions starting at $\Sigma_2$, a valid step since every orbit in $\mathcal{B}(\mathcal{A}_f)$ eventually crosses this cylinder. Notice also that, as the transition maps are linear with diagonal matrices, up to higher order terms 
the sequence of times $(t_j)_{j \in \mathbb{N}}$ depend essentially on the cylindrical coordinates $\rho$ and $z$. Figure~\ref{cross-sections2} summarizes the previous information.

\begin{figure}[ht]
\begin{center}
\includegraphics[height=7cm]{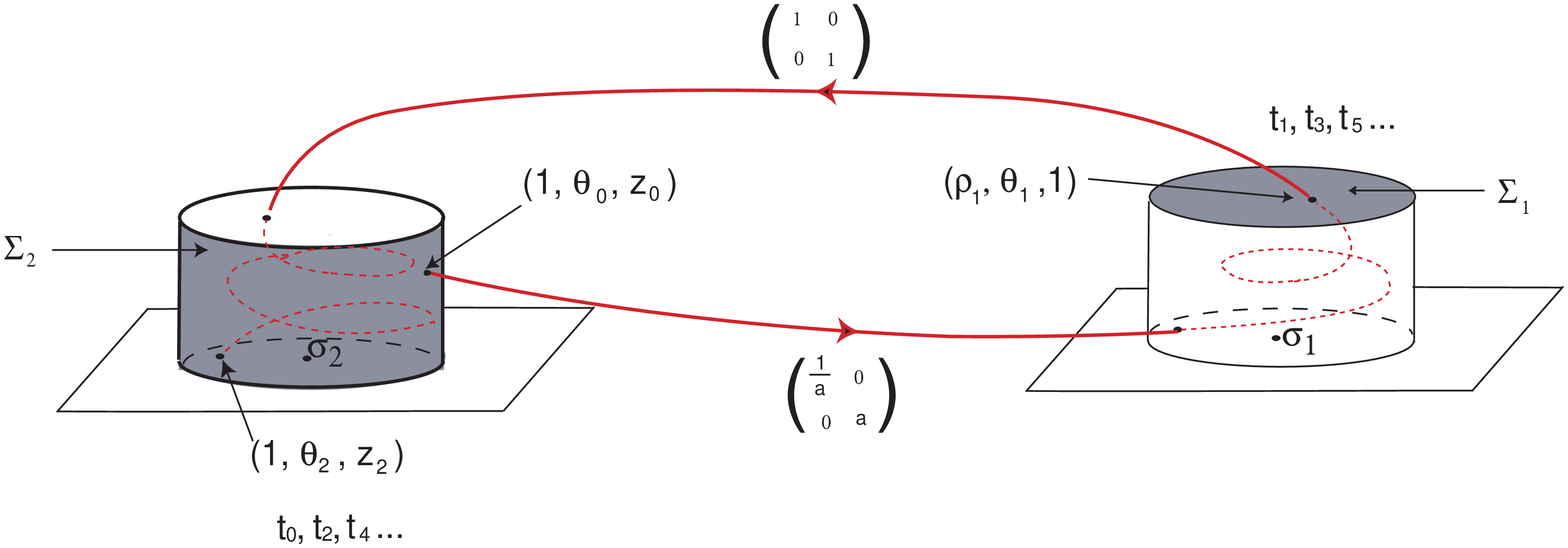}
\end{center}
\caption{\small The sequence $\left(t_i\right)_{i \, \in \, \NN_0}$ of hitting times for a point $(1, \,\theta_0,\,z_0) \in \Out^+(\sigma_2)$.}
\label{cross-sections2}
\end{figure}

\section{The invariants}\label{se:invariants}

In this section, we will examine how the hitting times sequences generate the set of invariants we will use. Starting with a point $P=(1,\theta_0,z_0) \in \Out^+(\sigma_2)$ at the time $t_0=0$, consider the sequences of times $\left(t_j\right)_{j\,\in\,\mathbb{N}}$ constructed in the previous subsection and define for each $i \in \NN_0=\NN\cup\{0\}$
\begin{equation}\label{eq:times}
\left\{\begin{array}{l}
P_{2i}:=\varphi(t_{2i}, \,\,P)=(1,\,\,\theta_{2i}, \,\,z_{2i}) \,\,\in\,\, \Out^+(\sigma_2) \\ \\

P_{2i+1}:=\varphi(t_{2i+1}, \,\,P)=(\rho_{2i+1},\,\,\theta_{2i+1},\,\,1) \in \Out^+(\sigma_1).
\end{array}
\right.
\end{equation}
We divide the trajectory $\left(\varphi(t,P)\right)_{t \,\in\, \RR^+_0}$ in periods of time corresponding to its sojourns in either $V_{\sigma_1}$ (that is, the differences ${t_{2i+1}-t_{2i}}$ for $i \in \NN_0$) or inside $V_{\sigma_2}$ (that is, ${t_{2i+2}-t_{2i+1}}$ for $i \in \NN_0$) during its travel paths that begin and end at $\Out^+(\sigma_2)$.

\begin{lemma}\label{le:calculus} Let $P=(1,\theta_0,z_0) \in \Out^+(\sigma_2)$ and take the defined sequence $(t_j)_{j\,\in\, \NN_0}$. Then:
\begin{enumerate}
\item $\lim_{i \to +\infty} \,({t_{2i+1}-t_{2i}})-\gamma_2\,({t_{2i}-t_{2i-1}}) = -\,\frac{1}{E_1}\log a.$
\bigbreak
\item $\lim_{i \to +\infty} \,({t_{2i+2}-t_{2i+1}})-\gamma_1\,({t_{2i+1}-t_{2i}}) = 0.$
\bigbreak
\item $\lim_{i \to +\infty}\, ({t_{2i+2}-t_{2i}})- \delta\,(t_{2i}-t_{2i-2}) = -\tau \,\log a.$
\bigbreak
\item\label{6.2(4)} For each $i \in \NN$, there exists $R_i \in \RR$ such that $\,\sum_{i=1}^\infty i\,|R_i|<\infty$ and
$$({t_{2i+2}-t_{2i}})- \delta\,(t_{2i}-t_{2i-2})= -\tau \,\log a + R_i.$$
\end{enumerate}
\end{lemma}

\medskip

\begin{proof} Using \eqref{eq:t1} and \eqref{eq:t2}, we may write
$$({t_{2i+1}-t_{2i}})-\gamma_2\,({t_{2i}-t_{2i-1}}) = -\,\frac{1}{E_1}\log\,\left(a\,\rho_{2i-1}^{\delta_2} + \mathcal{O}\left(\rho_{2i-1}^{\delta_2(1 + \varepsilon)}\right)\right) + \frac{\gamma_2}{E_2}\log\,(\rho_{2i-1}).$$
Summoning the fact that
\begin{eqnarray*}
\log \left(a \, \rho_{2i-1}^{\delta_2} + \mathcal{O}\left(\rho_{2i-1}^{\delta_2(1 + \varepsilon)}\right)\right) &=& \log \left(a \, \rho_{2i-1}^{\delta_2} \left(1 + \frac{\mathcal{O}\left(\rho_{2i-1}^{\delta_2 (1 + \varepsilon)}\right)}{a \, \rho_{2i-1}^{\delta_2}}\right)\right) \\
&=& \log a + \delta_2\,\log \rho_{2i-1} + \log \left(1 + \mathcal{O}\left(\rho_{2i-1}^{\delta_2 \,\varepsilon}\right)\right) \\
&=& \log a + \delta_2\,\log \rho_{2i-1} + \mathcal{O}\left(\rho_{2i-1}^{\delta_2 \,\varepsilon}\right)
\end{eqnarray*}
we deduce that
\begin{eqnarray*}
({t_{2i+1}-t_{2i}})-\gamma_2\,({t_{2i}-t_{2i-1}}) &=& -\,\frac{1}{E_1} \log a - \frac{C_2}{E_1 E_2} \log \,(\rho_{2i-1}) + \frac{C_2}{E_1 E_2} \log \,(\rho_{2i-1}) + \mathcal{O}\left(\rho_{2i-1}^{\delta_2 \, \varepsilon}\right) \\
&=& -\,\frac{1}{E_1} \log a + \mathcal{O}\left(\rho_{2i-1}^{\delta_2 \, \varepsilon}\right).
\end{eqnarray*}
Therefore, as $\lim_{i \,\to \,+\infty}\,\rho_{2i-1}^{\delta_2 \, \varepsilon}=0$ due to the asymptotic stability of $\mathcal{A}_f$, we get
$$\lim_{i \to +\infty} \,({t_{2i+1}-t_{2i}})-\gamma_2\,({t_{2i}-t_{2i-1}}) = -\,\frac{1}{E_1}\log a.$$

Similarly,
\begin{eqnarray}\label{eq:gamma1}
&&({t_{2i+2}-t_{2i+1}})-\gamma_1\,({t_{2i+1}-t_{2i}}) = -\,\frac{1}{E_2} \log \left((a\,z_{2i})^{\delta_1} +  \mathcal{O}\left((a\,z_{2i})^{\delta_1 (1 + \varepsilon)}\right)\right) - \frac{\gamma_1}{E_1} \log \,(a \,z_{2i})= \nonumber\\
&=& - \frac{C_1}{E_1 E_2} \log \,(a\,z_{2i}) + \frac{C_1}{E_1 E_2} \log \,(a\,z_{2i}) + \mathcal{O}\left((a\,z_{2i})^{\delta_1 \,\varepsilon}\right)
\end{eqnarray}
thus
\begin{equation*}
\lim_{i \to +\infty} \,({t_{2i+2}-t_{2i+1}})-\gamma_1\,({t_{2i+1}-t_{2i}}) = 0.
\end{equation*}

\medskip

Up to higher order terms (we will replace by dots), we have
\begin{eqnarray*}
{t_{2i+2}-t_{2i}} &=& -\,\frac{1}{E_1} \log \,(a\,z_{2i}) -\,\frac{1}{E_2} \log \,(\rho_{2i+1})+\ldots\\
&=& -\,\frac{1}{E_1} \log a  -\,\frac{1}{E_1} \log \left(\rho_{2i-1}^{\delta_2}\right) -\,\frac{1}{E_2} \log \left((a\, z_{2i})^{\delta_1}\right) + \ldots \\
&=& -\,\frac{1}{E_1} \log a  -\,\frac{{\delta_2}}{E_1} \log \left((a\, z_{2i-2})^{\delta_1}\right) -\,\frac{{\delta_1}}{E_2} \log  a\, -\,\frac{{\delta_1}}{E_2} \log \left(z_{2i}\right) + \ldots\\
&=& -\,\frac{1}{E_1} \log a  -\,\frac{{\delta}}{E_1} \log \left(a\, z_{2i-2}\right) -\,\frac{{\delta_1}}{E_2} \log  a\,  -\,\frac{{\delta_1}}{E_2} \log \left((\rho_{2i-1})^{\delta_2}\right) + \ldots\\
&=& -\,\frac{1}{E_1} \log a  -\,\frac{{\delta}}{E_1} \log  a\,   -\,\frac{{\delta}}{E_1} \log \left(z_{2i-2}\right) -\,\frac{{\delta_1}}{E_2} \log  a\,  -\,\frac{{\delta}}{E_2} \log \left((a\,z_{2i-2})^{\delta_1}\right) + \ldots\\
&=& -\,\frac{1}{E_1} \log a  -\,\frac{{\delta}}{E_1} \log  a\,   -\,\frac{{\delta}}{E_1} \log \left(z_{2i-2}\right) -\,\frac{{\delta_1}}{E_2} \log  a\,  -\,\frac{{\delta\delta_1}}{E_2} \log \left((a\,z_{2i-2})\right) + \ldots\\
&=& -\left[\frac{1}{E_1}+\frac{{\delta}}{E_1}  + \frac{{\delta_1}}{E_2} + \frac{{\delta\delta_1}}{E_2}\right] \log a  - \left[\frac{{\delta}}{E_1}+\frac{{\delta\delta_1}}{E_2}\right] \log \left(z_{2i-2}\right) + \ldots.\\
\end{eqnarray*}

On the other hand,
\begin{eqnarray*}
{t_{2i}-t_{2i-2}}&=& -\,\frac{1}{E_1} \log \,(a\,z_{2i-2}) -\,\frac{1}{E_2} \log \,(\rho_{2i-1}) + \ldots\\
&=& -\,\frac{1}{E_1} \log a  -\,\frac{1}{E_1} \log \left(z_{2i-2}\right) -\,\frac{1}{E_2} \log \left((a\,z_{2i-2})^{\delta_1}\right) + \ldots \\
&=& -\left[\frac{1}{E_1}+ \frac{{\delta_1}}{E_2}\right] \log a  -  \left[\frac{{\delta_1}}{E_2}  + \frac{1}{E_1}\right]\log \left(z_{2i-2}\right) + \ldots.
\end{eqnarray*}

Therefore,
$$\lim_{i \to +\infty} \,({t_{2i+2}-t_{2i}})-\delta\,({t_{2i}-t_{2i-2}})  = -\left[\frac{1}{E_1}+\frac{\delta_1}{E_2} \right] \log a = -\tau\, \log a.$$

\medskip

For each $i\in \NN_0$, let
$$T_i = t_{2i+2}-t_{2i}.$$
Then,
$$T_{i}-\delta \,T_{i-1} = ({t_{2i+2}-t_{2i}})- \delta\,(t_{2i}-t_{2i-2}) = -\tau \,\log a + \mathcal{O}((a\,z_{2i})^{\delta_1\,\varepsilon}).$$
So, if we define
$$i \in \NN\quad  \mapsto \quad R_i = \mathcal{O}((a\,z_{2i})^{\delta_1 \,\varepsilon})$$
then, as
\begin{eqnarray*}
\sqrt[i]{i\,\,|R_i|} &=& \sqrt[i]{i\,\,\left|\mathcal{O}((a\,z_{2i})^{\delta_1 \,\varepsilon})\right|} = \sqrt[i]{i}\,\,\left|\mathcal{O}\left((a\,z_{2i})^{\frac{\delta_1 \,\varepsilon}{i}}\right)\right| \\
\lim_{i \to +\infty} \,\sqrt[i]{i} &=& 1 \quad \quad \text{and} \quad \quad 0 < (a\,z_{2i})^{\delta_1 \,\varepsilon}  <  \mathcal{O}((a\,z_0)^i)
\end{eqnarray*}
we obtain
$$\limsup_{i \to +\infty} \,\sqrt[i]{i\,\,|R_i|} \leq a\,z_0 < 1.$$
Hence, using the root test, we conclude that the series $\sum_{i=1}^\infty \,i\,\,|R_i|$ converges.
\end{proof}

A straightforward computation gives additional information on the speed of convergence of the previous sequences and the connection between the return times sequences and the invariant $\omega_1 + \gamma_1\,\omega_2$.

\begin{corollary}\label{cor:speed-of-convergence} $\,$
\begin{enumerate}
\item $\lim_{i \to +\infty} \,\frac{t_{2i+2}\,-\,t_{2i+1}}{t_{2i+1}\,-\,t_{2i}} = \gamma_1$.
\bigbreak
\item $\lim_{i \to +\infty} \,\frac{t_{2i+1}\,-\,t_{2i}}{t_{2i}\,-\,t_{2i-1}} = \gamma_2$.
\bigbreak
\item $\lim_{i \to +\infty} \,\frac{t_{2i+2}\,-\,t_{2i}}{t_{2i}\,-\,t_{2i-2}} = \delta$.
\bigbreak
\item $\lim_{i \to +\infty}\, \frac{\omega_1\,\left(t_{2i+1}\,-\,t_{2i}\right) \,+ \,\omega_2\,\left(t_{2i+2}\,-\,t_{2i+1}\right)}{t_{2i+2}\,-\,t_{2i}} = \left(\omega_1 + \gamma_1\,\omega_2\right)\,\frac{1}{\gamma_1 \,+\, 1}.$
\end{enumerate}
\end{corollary}

\medskip

\begin{proof}
The first three items are immediate. Concerning the last one, it is enough to notice that
\begin{eqnarray*}
&&\frac{\omega_1\,\left(t_{2i+1}-t_{2i}\right) \,+ \,\omega_2\,\left(t_{2i+2}-t_{2i+1}\right)}{\left(t_{2i+2}-t_{2i+1}\right) \,+\, \left(t_{2i+1}-t_{2i}\right)} =\\
&=&\frac{\omega_1\,\left(t_{2i+1}-t_{2i}\right) \,+ \,\omega_2\,\left(t_{2i+2}-t_{2i+1}\right)}{t_{2i+1}-t_{2i}} \times \frac{t_{2i+1}-t_{2i}}{\left(t_{2i+2}-t_{2i+1}\right) \,+\, \left(t_{2i+1}-t_{2i}\right)}
 \end{eqnarray*}
and use the equalities
$$\lim_{i \to +\infty} \,\frac{t_{2i+2}\,-\,t_{2i+1}}{t_{2i+1}\,-\,t_{2i}} = \gamma_1 \quad \quad \text{and} \quad \quad \lim_{i \to +\infty} \,\frac{t_{2i+1}-t_{2i}}{\left(t_{2i+2}-t_{2i+1}\right) \,+\, \left(t_{2i+1}-t_{2i}\right)} = \frac{1}{\gamma_1 + 1}.$$
\end{proof}

\section{Proof of Proposition~\ref{thm:A}} The reasoning presented in this section follows Takens' proof of Theorem 1 in \cite{Takens1994}; a proof is included for the reader's convenience. Let $P \in \mathcal{B}(\mathcal{A}_f)$ and $\left(t_i\right)_{i \,\in \, \mathbb{N}_0}$ its hitting times sequence. Taking into account the estimates obtained in Section~\ref{se:invariants}, consider $\ell \in \mathbb{N}_0$ big enough so that the remainders in the computations to prove Lemma~\ref{le:calculus} are arbitrarily small. Then, for every continuous map $G\colon \EU \to \mathbb{R}$ we have
\begin{eqnarray*}
&&\frac{1}{t_{2\ell + 2i}- t_{2\ell}} \int_{t_{2\ell}}^{t_{2\ell + 2i}} G(\varphi(t,\,P))\,dt = \frac{1}{t_{2\ell + 2i}- t_{2\ell}} \,\sum_{k=0}^{i-1} \left[\int_{t_{2\ell + 2k}}^{t_{2\ell + 2k + 2}} G(\varphi(t,\,P))\,dt \right] = \\
&=& \frac{1}{t_{2\ell + 2i} - t_{2\ell}} \,\sum_{k=0}^{i-1} \left[\int_{t_{2\ell + 2k}}^{t_{2\ell + 2k + 1}} G(\varphi(t,\,P))\,dt + \int_{t_{2\ell + 2k + 1}}^{t_{2\ell + 2k + 2}} G(\varphi(t,\,P))\,dt \right]
\end{eqnarray*}
and so, applying Lemma~\ref{le:calculus}(2), we get
\begin{eqnarray*}
&&\lim_{i\to +\infty} \,\frac{1}{t_{2\ell + 2i} - t_{2\ell}} \int_{t_{2\ell}}^{t_{2\ell + 2i}} G(\varphi(t,\,P))\,dt = \\
&=&\lim_{i\to +\infty} \,\frac{1}{t_{2\ell + 2i}- t_{2\ell}} \sum_{k=0}^{i-1} \,\left[(t_{2\ell + 2k + 1}-t_{2\ell + 2k})\, G(\sigma_1) + (t_{2\ell + 2k + 2}-t_{2\ell + 2k + 1}) \,G(\sigma_2) \right]\\
&=& \lim_{i\to +\infty} \,\frac{1}{t_{2\ell + 2i}- t_{2\ell}} \sum_{k=0}^{i-1} \,\left[(t_{2\ell+2k+1}-t_{2\ell+2k})\, G(\sigma_1) + \gamma_1\,(t_{2\ell+2k+1}-t_{2\ell+2k})\,G(\sigma_2) \right]\\
&=& \lim_{i\to +\infty} \,\frac{1}{t_{2\ell + 2i}- t_{2\ell}} \sum_{k=0}^{i-1} \,\left(t_{2\ell+2k+1}-t_{2\ell+2k}\right)  \, \left[G(\sigma_1) + \gamma_1 \, G(\sigma_2)\right]\\
&=& \left[G(\sigma_1) + \gamma_1\, G(\sigma_2)\right]\,\,\lim_{i\to +\infty} \,\frac{1}{t_{2\ell + 2i}- t_{2\ell}} \sum_{k=0}^{i-1} \,\left(t_{2\ell+2k+1}-t_{2\ell+2k}\right).
\end{eqnarray*}
Additionally,
$$\left[G(\sigma_1) + \gamma_1\, G(\sigma_2)\right]\,\,\lim_{i\to +\infty} \,\frac{1}{t_{2\ell + 2i}- t_{2\ell}} \sum_{k=0}^{i-1} \,\left(t_{2\ell+2k+1}-t_{2\ell+2k}\right) = \frac{1}{1 + \gamma_1}\,G(\sigma_1) + \frac{\gamma_1}{1 + \gamma_1}\, G(\sigma_2)$$
because, during the period of time $[0, t_{2\ell + 2i}]$, the ratio between the time spent in the linearizing neighborhood of $\sigma_1$ an the total $t_{2\ell + 2i} - t_{2\ell}$ approaches $\frac{1}{1+\gamma_1}$ as $i$ goes to $+ \infty$. \\

We observe that we may replace $P$ by either $\sigma_1$ or $\sigma_2$ in the previous computation due to the fact that its orbit alternately approaches one of these two equilibria as the time goes to infinity, and $G$ is continuous. Moreover, while changing $(t_{2\ell + 2k + 2}-t_{2\ell+2k+1})$ into $\gamma_1\,(t_{2\ell+2k+1}-t_{2\ell+2k})$, we are omitting the higher order terms explicit in \eqref{eq:gamma1}; yet, this does not affect the limit of the averages $\left(\frac{1}{t_{2\ell+2i}-t_{2\ell}} \sum_{k=0}^{i-1} \,\left(t_{2\ell+2k+1}-t_{2\ell+2k}\right)\right)_{i \in \NN}$ since the remainder $\mathcal{O}\left((a\,z_{2\ell+2i})^{\delta_1 \,\varepsilon}\right)$ goes exponentially fast to zero (cf. \eqref{local map v}), so the corresponding series converges; besides, $\lim_{i\to +\infty} \,t_{2\ell+2i} -t_{2\ell} = +\infty$ (cf. Section~\ref{se:hitting-times}).\\

An analogous computation yields the second part of Proposition~\ref{thm:A}:
$$\lim_{i\to +\infty} \,\frac{1}{t_{2\ell+2i+1}-t_{2\ell}} \int_{t_{2\ell}}^{t_{2\ell+2i+1}} G(\varphi(t,\,P))\,dt = \frac{\gamma_2}{1 + \gamma_2} \,G(\sigma_1) + \frac{1}{1 + \gamma_2} \,G(\sigma_2).$$

Notice now that
\begin{eqnarray*}
&&\left(\frac{\gamma_2}{1 + \gamma_2} \,G(\sigma_1) + \frac{1}{1 + \gamma_2} \,G(\sigma_2)\right) - \left(\frac{1}{1 + \gamma_1} \,G(\sigma_1) + \frac{\gamma_1}{1 + \gamma_1} \,G(\sigma_2)\right) = \\
&=& \left(\frac{1 - \gamma_1\,\gamma_2}{(1 + \gamma_1)(1+\gamma_2)}\right) \, \left[G(\sigma_2) - G(\sigma_1)\right].
\end{eqnarray*}
So, if we choose a continuous map $G\colon \EU \to \mathbb{R}$ such that $G(\sigma_1) \neq G(\sigma_2)$ (whose existence is guaranteed by Urysohn's Lemma on the compact metric space $\EU$), then, as by assumption we have $\gamma_1\,\gamma_2 > 1$ (see Section~\ref{se:constants}), we conclude that
$$\frac{\gamma_2}{1 + \gamma_2} \,G(\sigma_1) + \frac{1}{1 + \gamma_2} \,G(\sigma_2) \neq \frac{1}{1 + \gamma_1} \,G(\sigma_1) + \frac{\gamma_1}{1 + \gamma_1} \,G(\sigma_2)$$
and therefore
$$\lim_{i\to +\infty} \,\frac{1}{t_{2\ell+2i}-t_{2\ell}} \int_{t_{2\ell}}^{t_{2\ell+2i}} G(\varphi(t,\,P))\,dt \neq \lim_{i\to +\infty} \,\frac{1}{t_{2\ell+2i+1}-t_{2\ell}} \int_{t_{2\ell}}^{t_{2\ell+2i+1}} G(\varphi(t,\,P))\,dt$$
confirming that the sequence of Birkhoff averages of $G$ along the orbit of $P$ does not converge.

\begin{remark}
As $\lim_{i \to +\infty} \,\frac{t_{2i+2}\,-\,t_{2i}}{t_{2i}\,-\,t_{2i-2}}=\delta$, the historic behavior in $\mathcal{B}(\mathcal{A}_f)$ is of type B1 according to the labeling proposed in \cite{JNY2009}.
\end{remark}

\section{Proof of Theorem~\ref{thm:B}}
We already know from \cite{P1978}, \cite{Takens1994}, \cite{Duf2001} and \cite{SSimo1989} that the numbers $\gamma_1$, $\gamma_2$, $\tau\,\log a$ and $\omega_1 + \gamma_1 \,\omega_2$ are invariants under topological conjugacy. We are left to prove that these invariants form a complete set. The argument we will present was suggested by Takens' ideas in \cite{Takens1994}, although we had to clarify several significant details of \cite{Takens1994} and make a few adjustments.

Let $f$ and $g$ be vector fields in $\mathfrak{X}^r_{\text{Byk}}(\EU)$ with invariants
$$\gamma_1, \quad \gamma_2, \quad \omega_1+\gamma_1\omega_2, \quad \tau \, \log a$$
and
$$\overline{\gamma}_1, \quad \overline{\gamma}_2, \quad \overline{\omega}_1+\overline{\gamma}_1 \, \overline{\omega}_2, \quad \overline{\tau} \,\log \overline{a}$$
respectively, and such that
\begin{equation}\label{eq:five-invariants}
\gamma_1 = \overline{\gamma}_1,\quad \gamma_2 = \overline{\gamma}_2, \quad \omega_1+\gamma_1\omega_2 = \overline{\omega}_1+\gamma_1\overline{\omega}_2, \quad \tau \, \log a = \overline{\tau} \,\log \overline{a}.
\end{equation}
We will show that these numbers enable us to construct a conjugacy between $f$ and $g$ in a neighborhood of the respective heteroclinic cycles $\mathcal{A}_f$ and $\mathcal{A}_g$. Notice that for a conjugacy between $f$ and $g$ to exist it is necessary that the conjugated orbits have hitting times sequences, with respect to fixed cross sections, that are uniformly close. With this information in mind, we will start associating to $f$ and any point $P$ in a fixed cross section $\Sigma$ another point $\widetilde{P}$ whose $f-$trajectory has a sequence of hitting times (at a possibly different but close cross section $\widetilde{\Sigma}$) which is determined by and uniformly close to the hitting times sequence of $P$, but is easier to work with (see, for instance, the computation \eqref{eq:theta}). This is done by slightly adjusting the cross section $\Sigma$ (which a topological conjugacy needs not to preserve) using the flow along the orbit of $P$; and then to find an injective and continuous way of recovering the orbits from the hitting times sequences. Afterwards, repeating this procedure with $g$ we find a point $Q$ whose $g-$trajectory has hitting times at some cross section equal to the ones of $\widetilde{P}$. Due to the fact that the invariants of $f$ and $g$ are the same, the map that sends $P$ to $Q$ is the aimed conjugacy. In the next subsections we will explain in detail this construction.

\subsection{A sequence of time intervals}\label{ssse:time1}
Fix $P = (\rho_0, \,\theta_0,\,z_0) \in \mathcal{B}(\mathcal{A}_f)$ and let $\left(t_i\right)_{i \,\in\, \NN_0}$ be the times sequence defined in \eqref{eq:times}. We start defining, for each $i \in \NN_0$, a finite family of numbers
$$\widetilde{T}_0^{\,\,(i)},\,\, \widetilde{T}_1^{\,\,(i)},\,\,  \widetilde{T}_2^{\,\,(i)},\, \ldots,\, \widetilde{T}_i^{\,\,(i)}$$
satisfying the following properties
\begin{eqnarray}\label{eq:Ttil-ij}
\widetilde{T}_i^{\,\,(i)} &=& T_i = t_{2i+2}-t_{2i} \nonumber \\
\widetilde{T}_j^{\,\,(i)} - \delta \,\widetilde{T}_{j-1}^{\,\,(i)} &=& -\tau\, \log a, \quad \quad \forall \, j \in \{1, 2, \ldots, i\}.
\end{eqnarray}
\medskip
For instance, as $T_i = t_{2i+2}-t_{2i}$ and $t_0=0$, we deduce from the previous equalities that
\begin{eqnarray*}
\widetilde{T}_0^{\,\,(0)} &=& T_0 = t_2 - t_0 = t_2 \\
\widetilde{T}_1^{\,\,(1)} &=& T_1 = t_4 - t_2, \qquad \widetilde{T}_0^{\,\,(1)} = \frac{\widetilde{T}_1^{\,\,(1)} + \tau \,\log a}{\delta} = \frac{T_1 + \tau \,\log a}{\delta}\\
\widetilde{T}_2^{\,\,(2)} &=& T_2 = t_6 - t_4, \qquad \widetilde{T}_1^{\,\,(2)} = \frac{T_2 + \tau \,\log a}{\delta}, \qquad \widetilde{T}_0^{\,\,(2)} = \frac{T_2 + (1 + \delta)\, \tau \,\log a}{\delta^2}\\
\widetilde{T}_3^{\,\,(3)} &=& T_3 = t_8 - t_6, \qquad  \widetilde{T}_2^{\,\,(3)} = \frac{T_3 + \tau\, \log a}{\delta}, \qquad \widetilde{T}_1^{\,\,(3)} = \frac{T_3 + (1 + \delta)\, \tau\, \log a}{\delta^2} \\
&&\qquad \qquad \qquad \qquad \qquad  \qquad \qquad \qquad \qquad \,\,\,\, \widetilde{T}_0^{\,\,(3)} = \frac{T_3 + (1 + \delta + \delta^2)\,\tau \log a}{\delta^3}.
\end{eqnarray*}
By finite induction, it is straightforward to generalize these examples and show that, for every $i \, \in \, \NN$,
\begin{equation}\label{eq:T0i}
\widetilde{T}_0^{\,\,(i)}=\frac{T_i + (\sum_{j=0}^{i-1} \,\delta^j) \, \tau\log a}{\delta^i}
\end{equation}
Therefore,

\begin{lemma}\label{Tempos1} For every $i \, \in \, \NN_0$, we have $\,\widetilde{T}_0^{\,\,(i+1)} - \,\widetilde{T}_0^{\,\,(i)} = \frac{R_{i+1}}{\delta^{i+1}}.$
\end{lemma}

\begin{proof} This is immediate after \eqref{eq:T0i} and Lemma~\ref{le:calculus}(4):
\begin{eqnarray*}
\widetilde{T}_0^{\,\,(i+1)} - \widetilde{T}_0^{\,\,(i)} & = & \frac{T_{i+1} + \tau \,\log a \left(\sum_{j=0}^{i}\delta^j\right)}{\delta^{i+1}}\,-\,\frac{T_i + \tau\, \log a \left(\sum_{j=0}^{i-1}\delta^j\right)}{\delta^i}\\
& = &  \frac{T_{i+1}-\delta T_i}{\delta^{i+1}} + \frac{\tau \,\log a \left(\sum_{j=0}^{i}\delta^j - \sum_{j=1}^{i}\delta^j\right)}{\delta^{i+1}} \\
& = &  \frac{T_{i+1}-\delta T_i}{\delta^{i+1}} + \frac{\tau\, \log a}{\delta^{i+1}}  =  \frac{R_{i+1}}{\delta^{i+1}}.
\end{eqnarray*}
\end{proof}

Now, Lemma~\ref{Tempos1} yields
\begin{eqnarray*}
\widetilde{T}_0^{\,\,(i)}& = & \widetilde{T}_0^{\,\,(i-1)} + \frac{R_i}{\delta^i} \\
& = &  \widetilde{T}_0^{\,\,(i-2)} +   \frac{R_{i-1}}{\delta^{i-1}} +\frac{R_i}{\delta^i} \\
& = &  \widetilde{T}_0^{\,\,(i-3)} +   \frac{R_{i-2}}{\delta^{i-2}}  + \frac{R_{i-1}}{\delta^{i-1}} + \frac{R_i}{\delta^i} \\
& \vdots & \\
& = &  \widetilde{T}_0^{\,\,(0)} +   \sum_{j=1}^i \frac{R_j}{\delta^{j}}
\end{eqnarray*}
As $\delta>1$, the series $\sum_{j=1}^\infty \frac{R_j}{\delta^{j}} $ converges, and so the sequence $\left(\widetilde{T}_0^{\,\,(i)}\right)_{i \,\in\, \NN_0}$ converges. Denote its limit by
\begin{equation}\label{T0}
\widetilde{T}_0:= \lim_{i \to + \infty}\, \widetilde{T}_0^{\,\,(i)} = T_0^{(0)} + \sum_{j=1}^\infty \frac{R_j}{\delta^{j}} = T_0 +   \sum_{j=1}^\infty \frac{R_j}{\delta^{j}}.
\end{equation}

\begin{remark}\label{re:importante}
A small change on the value of $t_2$ (or on any other time value $t_i$) with the same order of magnitude of $\mathcal{O}(z_0^c)$, for a positive constant $c$, does not alter the limit $\widetilde{T}_0$, because such a perturbation only requires us to consider a slightly different value for each $R_i$.
\end{remark}

\subsection{A sequence of adjusted hitting times}\label{ssse:time2}

For $i \geq 1$, consider the sequence $(\widetilde{T}_i)_{i \,\in\, \NN_0}$ satisfying
\begin{equation}\label{eq:Ttil}
\widetilde{T}_i = \delta \, \widetilde{T}_{i-1} - \tau\, \log a \qquad \forall \, i \, \in \, \NN
\end{equation}
where $\widetilde{T}_0$ was computed in \eqref{T0}. For example,
\begin{eqnarray*}
\widetilde{T}_1 &=& \delta \, \widetilde{T}_{0} - \tau\, \log a \\
\widetilde{T}_2 &=& \delta^2 \, \widetilde{T}_{0} - (1+\delta)\,\tau\, \log a \\
\widetilde{T}_3 &=& \delta^3 \, \widetilde{T}_{0} - (1+\delta+\delta^2)\,\tau\, \log a.
\end{eqnarray*}

\begin{lemma}\label{le:convergence} The series $\sum_{i=0}^{+\infty} \, (T_i - \widetilde{T}_{i})$ converges.
\end{lemma}

\begin{proof} We notice that, by construction, for every $i \, \in\,  \NN_0$ we have $|T_i -\widetilde{T}_i| \leq \sum_{j=i+1}^\infty |R_j|$. Consequently, for every $\ell \, \in\,  \NN_0$, we get
\begin{eqnarray*}
\sum_{i=0}^\ell \,|T_i - \widetilde{T}_{i}| & \leq & \sum_{i=0}^\ell \,\sum_{j=i+1}^\infty |R_j| \\
&=& \left(|R_1|+ |R_2|+ |R_3| \ldots \right)+ \left(|R_2|+ |R_3|+|R_4| \ldots \right)+ \left(|R_3|+ |R_4|+\ldots \right) + \ldots\\
&=& |R_1|+ 2\,|R_2|+ 3\,|R_3| + \ldots + i\,|R_i|+ \ldots \\
&=& \sum_{i=1}^\ell \,i\,|R_i|.
\end{eqnarray*}
As, according to Lemma~\ref{le:calculus}(4), the series $\sum_{i=1}^{+\infty} \,i\,|R_i|$ converges, the proof is complete. Notice that this result implies that $\lim_{i \to +\infty} (T_i - \widetilde{T}_i) = 0.$
\end{proof}

Lemma~\ref{le:convergence} ensures that we may take a sequence $\left(\widetilde{t}_{2i}\right)_{i \, \in \, \NN_0}$ of positive real numbers such that
\begin{eqnarray}\label{eq:times-even-til}
\widetilde{t}_{0} &=& 0  \nonumber\\
\widetilde{T}_i &=& \widetilde{t}_{2i+2}- \widetilde{t}_{2i} \nonumber\\
\lim_{i\, \to \, + \infty}\, (t_{2i} - \widetilde{t}_{2i}) &=& 0.
\end{eqnarray}
Moreover, by construction (see \eqref{eq:Ttil}) we have
\begin{equation}\label{eq:tau log a}
(\widetilde{t}_{2i+2}- \widetilde{t}_{2i}) - \delta\, (\widetilde{t}_{2i}- \widetilde{t}_{2i-2}) = - \tau\,\log a.
\end{equation}

Afterwards, we take a sequence $\left(\widetilde{t}_{2i+1}\right)_{i \, \in \, \NN_0}$ satisfying, for every $i \, \in \, \NN_0$,
\begin{equation}\label{eq:times-odd-til}
\widetilde{t}_{2i+2}- \widetilde{t}_{2i+1} = \gamma_1 \, (\widetilde{t}_{2i+1} - \widetilde{t}_{2i}).
\end{equation}
Therefore,
\begin{eqnarray}\label{eq:omega1 omega2}
\frac{\omega_1\,\left(\widetilde{t}_{2i+1}\,-\,\widetilde{t}_{2i}\right) \,+ \,\omega_2\,\left(\widetilde{t}_{2i+2}\,-\,\widetilde{t}_{2i+1}\right)}{\left(\widetilde{t}_{2i+2}\,-\,\widetilde{t}_{2i}\right)} &=& \frac{\omega_1\,\left(\widetilde{t}_{2i+1}\,-\,\widetilde{t}_{2i}\right) \,+ \,\gamma_1\,\omega_2\,\left(\widetilde{t}_{2i+1}\,-\,\widetilde{t}_{2i}\right)}{\left(\widetilde{t}_{2i+2}\,-\,\widetilde{t}_{2i+1}\right) + \left(\widetilde{t}_{2i+1}\,-\,\widetilde{t}_{2i}\right)} \nonumber \\
&=&\frac{\omega_1 + \gamma_1\,\omega_2}{\gamma_1 \,+\, 1}.
\end{eqnarray}

\bigskip

\begin{lemma}\label{le:prop-t-til} The following equalities hold:
\begin{enumerate}
\item $\lim_{i\, \to \, + \infty}\, (t_{2i+1} - \widetilde{t}_{2i+1}) = 0.$ \\
\item $\lim_{i \to +\infty} \,(\widetilde{t}_{2i+1} - \widetilde{t}_{2i}) - \gamma_2\,(\widetilde{t}_{2i} - \widetilde{t}_{2i-1}) = -\,\frac{1}{E_1}\log a.$\\
\end{enumerate}
\end{lemma}

\begin{proof}
Taking into account \eqref{eq:times-odd-til}, we have
\begin{equation*}
\widetilde{t}_{2i+1} - t_{2i+1} =  \frac{\widetilde{t}_{2i+2} + \gamma_1\, \widetilde{t}_{2i}}{1 + \gamma_1} - t_{2i+1} = \frac{\widetilde{t}_{2i+2} + \gamma_1\, \widetilde{t}_{2i} - t_{2i+1} - \gamma_1\, t_{2i+1}}{1 + \gamma_1}
\end{equation*}
and so, from \eqref{eq:times-even-til}, we conclude that
\begin{eqnarray*}
& &\lim_{i\, \to \, + \infty}\, (1 + \gamma_1)\left(\widetilde{t}_{2i+1} - t_{2i+1}\right) - \left[(t_{2i+2} - t_{2i+1}) - \gamma_1\, (t_{2i+1} - t_{2i})\right] = \\
&=& \lim_{i\, \to \, + \infty}\, \left(\widetilde{t}_{2i+2} + \gamma_1\, \widetilde{t}_{2i} - t_{2i+1} - \gamma_1 \,t_{2i+1}\right) - \left[(t_{2i+2} - t_{2i+1}) - \gamma_1\, (t_{2i+1} - t_{2i})\right] \\
&=& \lim_{i\, \to \, + \infty}\, (\widetilde{t}_{2i+2} - t_{2i+2}) + \gamma_1\, (\widetilde{t}_{2i} - t_{2i}) = 0.
\end{eqnarray*}
Therefore, using the information of Lemma~\ref{le:calculus}(2), we get
$$\lim_{i\, \to \, + \infty}\, \widetilde{t}_{2i+1} - t_{2i+1} = 0.$$

Concerning the second part of the lemma, notice that
\begin{eqnarray*}
&&\left[(\widetilde{t}_{2i+1} - \widetilde{t}_{2i}) - \gamma_2\,(\widetilde{t}_{2i} - \widetilde{t}_{2i-1})\right] - \left[(t_{2i+1} - t_{2i}) - \gamma_2\,(t_{2i} - t_{2i-1})\right] =\\
&=&  (\widetilde{t}_{2i+1} - t_{2i+1}) - (\widetilde{t}_{2i} - t_{2i}) - \gamma_2\,(\widetilde{t}_{2i} - \widetilde{t}_{2i}) + \gamma_2\,(\widetilde{t}_{2i-1} - t_{2i-1})
\end{eqnarray*}
thus, from Lemma~\ref{le:prop-t-til}(1), the Definition \eqref{eq:times-even-til} and Lemma~\ref{le:calculus}(1) we obtain
$$\lim_{i \to +\infty} \,(\widetilde{t}_{2i+1} - \widetilde{t}_{2i}) - \gamma_2\,(\widetilde{t}_{2i} - \widetilde{t}_{2i-1}) = \lim_{i \to +\infty} \, (t_{2i+1} - t_{2i}) - \gamma_2\,(t_{2i} - t_{2i-1}) = -\,\frac{1}{E_1}\log a.$$
\end{proof}

As we are mainly interested in conjugacies and the asymptotic behavior of the sequence $\left(t_i\right)_{i \,\in\, \NN_0}$, the construction of $\left(\widetilde{t}_{2i}\right)_{i \, \in \, \NN_0}$ must be independent of the starting difference $t_{2} - t_{0}$, and so irrespective of the size of the cross sections (with a smaller section we may miss the first few intersections of the orbits with the sections). More precisely,

\begin{lemma}\label{le:delay} Consider $N \in \NN_0$, the difference $T_N = t_{2N+2} - t_{2N}$ and the sequence $\left(\widetilde{T}_{0,N}^{\,\,(i)}\right)_{i \, \in \, \NN_0}$ defined as in \eqref{eq:Ttil-ij} but starting with $T_N$ instead of $T_0$. Then
$$\widetilde{T}_{0,N} := \lim_{i \to + \infty}\, \widetilde{T}_{0,N}^{\,\,(i)} = \widetilde{T}_N.$$
\end{lemma}

\begin{proof} Consider $N=1$ and observe that, according to \eqref{eq:T0i},
$$\widetilde{T}_{0,1}^{\,\,(i)} = \frac{T_{i+1} + (\sum_{j=0}^{i-1} \,\delta^j) \, \tau\,\log a}{\delta^i}$$
so
$$\widetilde{T}_{0,1}^{\,\,(i+1)} - \widetilde{T}_{0,1}^{\,\,(i)} = \frac{R_{i+2}}{\delta^{i+1}}.$$
Therefore
$$\widetilde{T}_{0,1} = T_1 + \sum_{j=1}^{+ \infty} \, \frac{R_{j+1}}{\delta^{j}} = \delta\,\widetilde{T}_0 - \tau\,\log a.$$
By induction in $N$, we get
$$\widetilde{T}_{0,N} = T_N + \sum_{j=1}^{+ \infty} \, \frac{R_{j+N}}{\delta^{j}} = \delta^N\,\widetilde{T}_0 - \left(\sum_{j=0}^{N-1} \,\delta^j\right) \, \tau\,\log a = \widetilde{T}_N.$$
\end{proof}

Consequently, up to a shift of the indices from $i$ to $i - 2N$, we obtain the same sequence $\left(\widetilde{t}_{i}\right)_{i \, \in \, \NN_0}$ if we build it using the equalities in \eqref{eq:times-even-til} and \eqref{eq:times-odd-til} but starting with $t_{2N}=0$ and $\widetilde{T}_N$ instead of $t_0=0$ and $\widetilde{T}_0$.

\subsection{Realization of the sequence of times}\label{sse:realization}

As any solution of $f$ in $\mathcal{B}(\mathcal{A}_f)$ eventually hits $\Out\,(\sigma_2)$, we may apply the previous construction to all the orbits of $f$ in $\mathcal{B}(\mathcal{A}_f)$. So, given any $P_0 \in \mathcal{B}(\mathcal{A}_f)$, we take the first non-negative hitting time of the forward orbit of $P_0$ at $\Out\,(\sigma_2)$, defined by
$$t_{\Sigma_2} (P_0) = \min \,\{t \in \RR^+_0 \colon \varphi(t,\,P_0) \in \Out\,(\sigma_2)\}.$$
As $\Out^+(\sigma_2)$ and $\Out^-(\sigma_2)$ are relative-open sets, this first-hitting-time map is continuous with $P_0$. Afterwards, given
$$P = \varphi(t_{\Sigma_2}(P_0),\,P_0) = \left(1,\,\theta_0, z_0\right) \in \Out\,(\sigma_2)$$
we consider its hitting times sequence $\left(t^{\,\,(P)}_{i}\right)_{i \, \in \, \NN_0}$
and build the sequence $\left(\widetilde{t}^{\,\,\,(P)}_{i}\right)_{i \, \in \, \NN_0}$ as explained in the previous section.

We now proceed as follows. Adjusting the cross sections $\Sigma_1$ and $\Sigma_2$ 
we find a point $\widetilde{P} \in \Out\,(\sigma_2)$ in the $f-$trajectory of $P$ whose hitting times sequence is precisely $\left(\widetilde{t}^{\,\,\,(P)}_{i}\right)_{i \, \in \, \NN_0}$. 
Notice that the new cross sections are close to the previous ones since the sequences $\left(t_{i}\right)_{i \, \in \, \NN_0}$ and $\left(\widetilde{t}_{i}\right)_{i \, \in \, \NN_0}$ are uniformly close (cf. \eqref{eq:times-even-til} and Lemma~\ref{le:prop-t-til}(1)). We are left to show that there exists only one such trajectory with hitting times sequence $\left(\widetilde{t}^{\,\,\,(P)}_{i}\right)_{i \, \in \, \NN_0}$.

\subsubsection*{\textbf{\emph{Uniqueness of $\widetilde{P}$}}}\label{sse:special-point}

Given a sequence of times $\left(\widetilde{t}_i\right)_{i \, \in \, \NN_0}$ satisfying  $\widetilde{t}_0=0$, Lemma~\ref{le:prop-t-til} and the properties \eqref{eq:times-even-til}, \eqref{eq:tau log a}, \eqref{eq:times-odd-til} and \eqref{eq:omega1 omega2}, one may recover from its terms the coordinates of a point $(1, \,\theta_0, \,z_0) \in \Out^+(\sigma_2)$ whose $i$th hitting time is precisely $\widetilde{t}_i$. Firstly, we solve the equation (see \eqref{eq:t1})
\begin{equation}\label{eq:z}
\widetilde{t}_1 = -\,\frac{1}{E_1} \log \,(a\,z_0)
\end{equation}
obtaining $z_0$. Then, using \eqref{eq:t2}
\begin{equation}\label{eq:rho}
\widetilde{t}_2 = \widetilde{t}_1 -\,\frac{1}{E_2} \log \,(\rho_1)
\end{equation}
we compute $\rho_1$. And so on, getting from such a sequence of times all the values of the radial and height cylindrical coordinates $\left(\rho_{2i+1}\right)_{i \, \in \, \NN_0}$ and $\left(z_{2i}\right)_{i \, \in \, \NN_0}$ of the successive hitting points at $\Out^+(\sigma_1)$ and $\Out^+(\sigma_2)$, respectively. Not knowing an explicit expression for the function $\mathcal{S}_1$, however, nothing has been disclosed about $\theta_0$ from these computations.

Concerning the evolution in $\mathbb{R}^+$ of the angular coordinates, the spinning in average inside the cylinders is given, for every $i \in \mathbb{N}_0$, by
\begin{eqnarray}\label{eq:spin}
\frac{\theta_{2i+2} - \frac{1}{a}\,\theta_{2i}}{\widetilde{t}_{2i+2}-\widetilde{t}_{2i}} &=& \frac{(\theta_{2i+2} - \theta_{2i+1}) + (\theta_{2i+1} - \frac{1}{a}\,\theta_{2i})}{\widetilde{t}_{2i+2}-\widetilde{t}_{2i}} \nonumber\\
&=& \frac{\omega_2\,(\widetilde{t}_{2i+2} - \widetilde{t}_{2i+1}) + \omega_1\,(\widetilde{t}_{2i+1} - \widetilde{t}_{2i})}{\widetilde{t}_{2i+2}-\widetilde{t}_{2i}} \nonumber\\
&=& \frac{\omega_1 + \gamma_1\, \omega_2}{\gamma_1 + 1}
\end{eqnarray}
(cf. \eqref{local map v} and \eqref{eq:omega1 omega2}).
Moreover, \eqref{eq:times-odd-til} indicates that
$$\frac{\theta_{2i+1} - \frac{1}{a}\,\theta_{2i}}{\theta_{2i+2}-\theta_{2i+1}} = \frac{\omega_1\,(\widetilde{t}_{2i+1} - \widetilde{t}_{2i})}{\omega_2\,\left(\widetilde{t}_{2i+2} - \widetilde{t}_{2i+1}\right)} = \frac{\omega_1}{\gamma_1 \,\omega_2}.$$
So
\begin{eqnarray*}
\theta_{2i+2} - \theta_{2i} &=& (\theta_{2i+2} - \theta_{2i+1}) + \left(\theta_{2i+1} - \frac{1}{a}\,\theta_{2i}\right) + \frac{1}{a}\,\theta_{2i}\\
&=& (\theta_{2i+2} - \theta_{2i+1})\,\left(\frac{\omega_1}{\gamma_1 \,\omega_2} + 1\right) + \frac{1}{a}\,\theta_{2i}\\
&=& \omega_2\,(\widetilde{t}_{2i+2} - \widetilde{t}_{2i+1})\,\left(\frac{\omega_1}{\gamma_1 \,\omega_2} + 1\right) + \frac{1}{a}\,\theta_{2i} \\
&=& \frac{\omega_1 + \gamma_1\, \omega_2}{\gamma_1}\,(\widetilde{t}_{2i+2} - \widetilde{t}_{2i+1}) + \frac{1}{a}\,\theta_{2i}.
\end{eqnarray*}
On the other hand, from \eqref{eq:spin} we get
\begin{eqnarray*}
\theta_{2i+2} - \theta_{2i} &=& \left(\theta_{2i+2} - \frac{1}{a}\,\theta_{2i}\right) + \theta_{2i}\,\left(\frac{1}{a} - 1\right) \\
&=& \frac{\omega_1 + \gamma_1\, \omega_2}{\gamma_1 + 1}\,(\widetilde{t}_{2i+2}-\widetilde{t}_{2i}) + \theta_{2i}\,\left(\frac{1}{a} - 1\right).
\end{eqnarray*}
Consequently,
\begin{equation*}
\frac{\omega_1 + \gamma_1\, \omega_2}{\gamma_1}\,(\widetilde{t}_{2i+2} - \widetilde{t}_{2i+1}) + \frac{1}{a}\,\theta_{2i} = \frac{\omega_1 + \gamma_1\, \omega_2}{\gamma_1 + 1}\,(\widetilde{t}_{2i+2}-\widetilde{t}_{2i}) + \theta_{2i}\,\left(\frac{1}{a} - 1\right)
\end{equation*}
from which the angular coordinate $\theta_{2i}$ is uniquely determined:
\begin{eqnarray}\label{eq:theta}
\theta_{2i} &=& \left(\omega_1 + \gamma_1\, \omega_2\right)\,\left[\frac{\widetilde{t}_{2i+2} - \widetilde{t}_{2i}}{\gamma_1 + 1} - \frac{\widetilde{t}_{2i+2} - \widetilde{t}_{2i+1}}{\gamma_1}\right].
\end{eqnarray}
In particular, we confirm that there is a unique solution $\theta_0$.

\subsection{The conjugacy between $f$ and $g$}
Let $\overline{\sigma}_1$ and $\overline{\sigma}_2$ be the two hyperbolic saddle-foci of $g$ whose eigenvalues are, respectively,
\begin{eqnarray*}
-\,\overline{C}_{1} \pm \overline{\omega}_1 \,i &\quad \quad \text{and}  \quad \quad & \quad \overline{E}_{1} \\
\overline{E}_{2} \pm \overline{\omega}_2 \,i &\quad  \quad \text{and}  \quad \quad & -\,\overline{C}_{2} \nonumber
\end{eqnarray*}
where
$$\overline{\omega}_1 > 0, \quad \overline{\omega}_2 > 0, \quad \overline{C}_1 > \overline{E}_1 > 0 \quad \text{and} \quad \overline{C}_2 > \overline{E}_2 > 0$$
whose values define the invariants of $g$ which, by assumption, satisfy the equalities \eqref{eq:five-invariants}.

Consider linearizing neighborhoods of $\overline{\sigma_1}$ and $\overline{\sigma_2}$, with the corresponding cylindrical coordinates, and take a point $P=(1,z_0,\theta_0) \in \Sigma_2 \cap \Out^+(\sigma_2)$, the corresponding hitting times sequence $\left(t_i\right)_{i \, \in \, \NN_0}$ at cross sections $\Sigma_1$ and $\Sigma_2$ and the sequence of times $\left(\widetilde{t}_i\right)_{i \, \in \, \NN_0}$ obtained in Subsection~\ref{ssse:time2}. Next, we find a unique point $Q_P$, given in local coordinates by $(1,\overline{z_0}, \overline{\theta_0})$, as done for $f$ in Subsection~\ref{sse:special-point} using estimates similar to \eqref{eq:z}, \eqref{eq:rho} and \eqref{eq:theta}:
\begin{eqnarray}\label{eq:point-for-g}
\overline{z_0} &=& \frac{e^{-\overline{E_1}\,\widetilde{t}_1}}{\overline{a}} \\
\overline{\rho_1} &=& e^{-\overline{E_2}\,\left(\widetilde{t}_2 - \widetilde{t}_1\right)} \nonumber\\
\overline{\theta_{0}} &=& \left(\omega_1 + \gamma_1\, \omega_2\right)\,\left[\frac{\widetilde{t}_{2}}{\gamma_1 + 1} - \frac{\widetilde{t}_{2} - \widetilde{t}_{1}}{\gamma_1}\right]\label{eq:theta-g}.
\end{eqnarray}
The set of these points build cross sections $\overline{\Sigma_1}$ and $\overline{\Sigma_2}$ for $g$ at which the points $Q_P$ have the prescribed hitting times $\left(\widetilde{t}_i\right)_{i \, \in \, \NN_0}$ by the action of $g$. Afterwards, we take the map
$$H: P \in \Sigma_2 \cap \Out^+(\sigma_2) \quad \mapsto \quad Q_P$$
and extend it using the flows $\varphi$ and $\overline{\varphi}$ of $f$ and $g$, respectively: for every $t \in \RR$, set $H(\varphi_t(P))=\overline{\varphi}_t(H(P)).$ An analogous construction is repeated for $\Out^-(\sigma_2)$.

\begin{lemma}
$H$ is a conjugacy.
\end{lemma}

\begin{proof}
Firstly, given two initial points $P_1 \neq P_2$ in $\Sigma_2 \cap \Out^+(\sigma_2)$, their hitting times sequences (cf. \eqref{eq:t1} and \eqref{eq:t2}) are not only different but not even uniformly close due to the expanding components of the saddle-type dynamics in the linearizing neighborhoods of $\sigma_1$ and $\sigma_2$. Therefore, $P_1$ and $P_2$ are mapped under $H$ into different points. So $H$ is injective.

Notice that, for this conclusion, it is essential that $\mathcal{A}_f$ and $\mathcal{A}_g$ are global attractors, ensuring that the $\rho$ and $z$ coordinates decrease to $0$ as time goes to $+ \infty$ along the orbits of initial conditions different from the equilibria. This in turn implies that:
(i) the time deviations (which are expressed by the maps $\mathcal{S}_j$ and $\mathcal{T}_j$ in the formulas of Subsection~\ref{sse:local saddle-foci}) from the time estimates done in Section~\ref{se:hitting-times} are asymptotically arbitrarily small, and so their impact is negligible;
(ii) the time increments caused by the twisting around the equilibria (cf item (4) of Corollary~\ref{cor:speed-of-convergence}) are the same for $f$ and $g$ because $\omega_1 + \gamma_1\,\omega_2 = \overline{\omega_1} + \gamma_1\,\overline{\omega_2}$, and so their impact may be discarded.
 We also remark that, if we repeat the previous construction starting with $g$ instead of $f$, we obtain a map that must be the inverse of $H$. And, indeed, so it is since $\omega_1 + \gamma_1\,\omega_2 = \overline{\omega_1} + \gamma_1\,\overline{\omega_2}$, because the angular deviations, due to the twisting around the equilibria and which intervene in the computation \eqref{eq:theta-g}, are the same for $f$ and $g$.

Secondly, if $P_1$ and $P_2$ are close enough, then the first terms of the corresponding hitting times sequences are sufficiently near to ensure that $\widetilde{P_1}$ is arbitrarily close to $\widetilde{P_2}$ (cf. \eqref{eq:point-for-g} and \eqref{eq:theta-g}). Thus $H$ is continuous in $\Sigma_2 \cap \Out^+(\sigma_2)$, and its extension to $\mathcal{B}(\mathcal{A}_f)$ is continuous by definition. This ends the proofs of the lemma and Theorem~\ref{thm:B}.

Finally, we observe that the conjugacy $H$ extends to $\mathcal{A}_f$.

\end{proof}

\end{document}